\documentclass[english]{amsart}
\usepackage[english]{babel}
\usepackage{bbm}
\usepackage{verbatim}

\usepackage{amsfonts,amssymb,mathtools}
\usepackage{amsmath}

\usepackage{xcolor}

\usepackage{graphicx,xcolor}
\graphicspath{{Fig/}}
\usepackage{float}
\usepackage{tikz}
\usepackage{hyperref}

\usepackage[style=numeric,maxcitenames=2,maxbibnames=10,doi=false,isbn=false,url=false,natbib=true,giveninits=true,hyperref,bibencoding=utf8,backend=biber]{biblatex}
\AtEveryBibitem{%
  \clearfield{note}%
  \clearfield{issn}%
  \clearlist{language}%
  }

\newcommand{\N}{\ensuremath{\mathbb{N}}}
\newcommand{\R}{\ensuremath{\mathbb{R}}}

\newcommand{\Z}{\ensuremath{\mathbb{Z}}}
\newcommand{\E}{\ensuremath{\mathbb{E}}}
\renewcommand{\P}{\ensuremath{\mathbb{P}}}

\newcommand{\ind}[1]{\ensuremath{\mathbbm{1}_{\left\{#1\right\}}}}
\newcommand{\diff}{\mathop{}\mathopen{}\mathrm{d}}
\newcommand{\cal}[1]{\ensuremath{\mathcal{#1}}}
\newcommand\croc[1]{\left\langle #1\right\rangle}
\newcommand\steq[1]{\stackrel{\text{\rm #1.}}{=}}

\newcommand\Pois[1]{{\rm Pois}\left[#1\right]}
\newcommand\e[1]{e_{{\scriptscriptstyle{#1}}}}

\def\eps{\varepsilon}
\def\cadlag{c\`adl\`ag }

\def\krs{\kappa_{\!\scriptscriptstyle{R}\!\scriptscriptstyle{S}}}
\def\ksr{\kappa_{\!\scriptscriptstyle{S}\!\scriptscriptstyle{R}}}
\def\klr{\kappa_{\scriptscriptstyle{L}\!\scriptscriptstyle{R}}}
\def\kq0{\kappa_{\!\scriptscriptstyle{Q}\!\scriptscriptstyle{0}}}
\def\k0q{\kappa_{\!\scriptscriptstyle{0}\!\scriptscriptstyle{Q}}}
\def\kri{\kappa_{\!\scriptscriptstyle{R}\!\scriptscriptstyle{I}}}
\def\kil{\kappa_{\!\scriptscriptstyle{I}\!\scriptscriptstyle{L}}}
\def\kqu{\kappa_{\!\scriptscriptstyle{Q}\!\scriptscriptstyle{U}}}

\newtheorem{proposition}{Proposition}
\newtheorem{definition}[proposition]{Definition}

\newtheorem{theorem}[proposition]{Theorem}
\newtheorem{corollary}[proposition]{Corollary}

\title[Stochastic Models of Regulation]{Stochastic Models of Resource Allocation in Chemical Reaction Networks}

\date{\today}
\author[V. Fromion]{Vincent Fromion}
\email{Vincent.Fromion@inrae.fr}
\address[V.~Fromion, J. Zaherddine]{INRAE, MaIAGE, Université Paris-Saclay, Domaine de Vilvert, 78350 Jouy-en-Josas, France}
\author[Ph. Robert]{Philippe Robert}
\email{Philippe.Robert@inria.fr}
\urladdr{http://www-rocq.inria.fr/who/Philippe.Robert}
\author[J. Zaherddine]{Jana Zaherddine}
\email{Jana.Zaherddine@inria.fr}
\address[Ph.~Robert, J. Zaherddine]{INRIA Paris, 48, rue Barrault, CS 61534, 75647 Paris Cedex, France}


\begin{document}
\maketitle

\begin{abstract}
This paper analyses of  a stochastic model of a chemical reaction network with three types of chemical species  ${\cal R}$, ${\cal M}$ and ${\cal U}$ that interact to transform a flow of external resources, the chemical species ${\cal Q}$, to produce a product, the chemical species ${\cal P}_r$. A regulation mechanism involving the sequestration of the chemical species ${\cal R}$ when the flow of resources is too low is investigated.  The  original motivation  of the study is of  analyzing the qualitative properties of  a key regulation mechanism of gene expression in biological cells, the {\em stringent response}. 

A scaling analysis of a Markov process in $\N^5$  representing the state of the chemical reaction network is achieved. It is shown that, depending on the parameters of the model, there are, quite surprisingly, three possible asymptotic regimes. To each of them corresponds a stochastic averaging principle with  a fast process expressed in terms of a network of $M/M/\infty$ queues. One of these regimes, the optimal sequestration regime, does not seem to have been identified up to now. Under this regime, the input flow of resources is low but the state of the network is still acceptable in terms of unused macro-molecules, showing  the remarkable efficiency of this regulation mechanism. The technical proofs of the main convergence results  rely on a combination of coupling arguments, technical estimates of the solutions of SDEs,  of sample paths of fast processes in particular, and the stability properties of some dynamical systems in $\R^2$. 
\end{abstract}

\maketitle

\bigskip

\hrule

\vspace{-3mm}

\tableofcontents

\vspace{-1cm}

\hrule

\section{Introduction}
This paper proposes a scaling analysis of  a stochastic model of a chemical reaction network which can be described as follows. See Section~\ref{Trans-Intro} for a detailed description.

There are two chemical species $R$-particles and $M$-particles which can form via pairing mechanisms a  complex ${\cal RM}_I$ which becomes after some steps a functional complex ${\cal RM}_L$.

There is an incoming flow of chemical species, $Q$-particles, resources, which can be paired to a chemical species, $U$-particles, defined as suppliers. Such a pair can be processed by a ${\cal RM}_L$-type complex in order to obtain the final product, a $P$-particle. At this point, the corresponding $R$-, $M$-, and $U$-particles are once again free in the environment.

Overall, the flow of $Q$-particles is used to create $P$-particles via a set of chemical reactions driven by random events. The regulation mechanism considered in this paper  involves the sequestration of $R$-particles. This control is designed to happen when the flow of $Q$-particles is too low, to avoid having too many unused complexes of the type ${\cal RM}_L$. 

The  initial motivation  of our study is of  investigating the qualitative properties of  an important regulation mechanism of gene expression in biological cells, the {\em stringent response}. See~\citet{Bouveret} and~\citet{Traxler}.  This type of chemical network can be used to describe several important processes of molecular biology as we will see. We first give a quick reminder of the gene expression context. 

\subsection{Gene Expression}
In bacterial cells, protein production is a central process. The production of macro-molecules, such as polymerases, ribosomes, and RNAs, as well as molecules like amino acids and energy sources, consumes most of cell resources. This enables the production of a diverse set of proteins with various functional properties.

It  can  represented as a two-step mechanism:

The {\em transcription step}.  Macro-molecules {\em polymerases} bind to  genes of DNA to produce  different types of RNAs: mRNAs, rRNAs, tRNAs, sRNAs, \ldots 

The {\em translation step}. A couple of macro-molecules $(R,M)$, where $R$ is a  {\em ribosome} and $M$ is an mRNA, may bind  to form a complex $RM_I$. After this {\em initiation phase}, the complex is transformed into a complex $RM_L$ for the {\em elongation phase}: A chain of amino-acids corresponding to the protein associated to this mRNA is sequentially built. Loosely speaking, an amino-acid $Q$, a biological brick of proteins, is paired  to an uncharged/free {\em transfer RNA} (tRNA),  $U$, to form a complex $UQ$, a charged tRNA. If the complex $RM_L$ is waiting for this type of amino-acid, the complex $UQ$ can be bound/processed by $RM_L$ so that $Q$ is added to the current chain of amino-acids being built. This procedure goes on  until the desired protein is completed and can be detached from the complex $RM_L$ and the complex itself breaks  into  $R$  and $M$ macro-molecules.  

\subsection*{Regulation Mechanisms}
The growth of a bacterial population in a given medium leads to an active consumption of resources necessary to produce new cells. When the concentrations of different resources in the medium  are high enough  for some time, the bacterium use them efficiently, via its complex regulatory system, to  reach a steady growth regime with a fixed  growth rate. 

When resources are scarce, a bacterium has to adapt,  to either exploit differently the available resources, or to do without some of them, as for example when some amino-acids are missing. For many classes of bacteria  resources maximizing the growth rate are used in priority. In the context of this adaptation, when resources are decaying,  a bacterial cell  has to decrease its growth rate and even  to ultimately stop its growth. For a global deterministic model of allocation in bacterial cells, see~\citet{GFS} and~\citet{GMCJ}. 

%%%%
\begin{center}
\begin{figure}[H]
\resizebox{12cm}{!}{%
\begin{tikzpicture}[->,node distance=8mm]

\node[black, very thick,circle,draw](R_S) at (2,0){${\cal R}_S$};
\node[black, very thick,circle,draw](R_F) at (6,0){${\cal R}$};
\node[black, very thick, circle,draw](R_M) at (10.5,0){${\cal RM}_I$};
\node[black, very thick, circle,draw](R_1) at (14,0){${\cal RM}_L$};
\node[black, very thick, circle,draw](P) at (15,-2){$\,{\cal P}_r\,$};

    \node () [black,very thick,rectangle,draw]at (1.8,1){{\bf Sequestration }};
    \node () [black,very thick,rectangle,draw]at (5.8,1){{\bf Free }};
    \node () [black,very thick,rectangle,draw]at (10.3,1){{\bf Initiation }};
    \node () [black,very thick,rectangle,draw]at (13.9,1){{\bf Elongation }};
    \node () [black]at (0,-0){{\sc $R$-particles }};
    \node () [black]at (0,-0.4){{\rm\small{(Ribosomes)} }};

    \node () [black]at (16.5,-2){{\sc Product }};
    \node () [black]at (16.5,-2.4){{\rm\small{(Proteins)} }};
          
  \path (R_S) edge [black,very thick,left,midway,above] node {$\ksr S$}(R_F);
  \path (R_F) edge [black,very thick,bend left=45,midway,above] node {$\krs RU $} (R_S);
  \path (R_F) edge [black,very thick,left,midway,above] node {$\kri \scriptscriptstyle{(M^0_N{-}(N{-}R{-}S))R}$} (R_M);
  \path (R_M) edge [black,very thick,left,midway,above] node {$\kil \scriptscriptstyle{(N{-}R{-}S{-}L)}$} (R_1);
  \path (R_1) edge [black,very thick,bend left=45,pos=0.15,above left] node {$\klr L(U^0_N{-}U)$} (R_F);
  \path (R_1) edge [black,very thick,bend left=-70,below] node {} (P);

  \node[black, very thick, circle,draw](Q1) at (1,-4.5){${\cal Q}$};
  \node[black, very thick, circle,draw](0) at (5,-4.5){$\emptyset$};

  %  \path (Q1) edge [black,very thick,bend left=25,midway,above] node {$(\kqu U{+}\kq0)Q$} (0);
    \path (Q1) edge [black,very thick,bend left=25,midway,above] node {$\kq0 Q$} (0);
  \path (0) edge [black,very thick,bend left=25,midway,below] node {$\k0q N $} (Q1);
  \node () [black]at (-1,-3.5){{\sc Resources }};
  \node () [black]at (-1,-3.9){{\rm\small{(Amino-Acids)} }};
  
  \node[black, very thick, circle,draw](0) at (9,-4.5){${\cal U}$};
  \node[black, very thick, circle,draw](T1) at (12,-4.5){${\cal UQ}$};
  \path (T1) edge [black,very thick,bend left=25,midway,below] node {$ \klr L(U^0_N{-}U)$} (0);
  \path (0) edge [black,very thick,bend left=25,midway,above] node {$\kqu UQ$} (T1);

  \node () [black]at (7,-3.5){{\sc Transporters }};
  \node () [black]at (7,-3.9){{\rm\small{(tRNAs)} }};

\end{tikzpicture}}
\caption{Transitions of Chemical Species}\label{Fig1}
\end{figure}
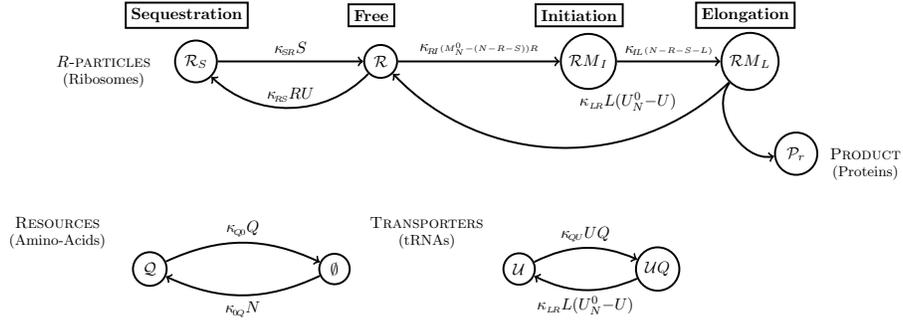
\end{center}

In~\citet{FRZ} we have analyzed a stochastic model of the regulation of the transcription phase. A macro-molecule, a 6S-RNA can sequester a subset of polymerases and therefore reduce the production of RNAs and consequently of proteins. 

In this paper, we are studying the regulation of the translation phase.  There are analogies with the regulation of transcription. There is also a sequestration mechanism (for ribosomes). But it is done via a more sophisticated system. We give a quick sketch of its basic principle. When a flow of amino-acids is too low, there will be many complexes of the type $RM_L$ waiting for this missing amino-acid. If this situation lasts, it may end-up  that  a ``wrong'' amino-acid  is  used instead,  leading to  an error for the composition of the protein and therefore, in some cases, a non-functional protein.

For this reason, the cell has to produce a maximal number of proteins to sustain its growth rate,  but at the same time the number of complexes $RM_L$ waiting for an amino-acid should not be too large. 

When there is a starvation of some types of amino-acids, a significant fraction of the corresponding tRNAs are uncharged, i.e. not paired with an amino-acid. Recall that they ``bring'' amino-acids to complexes of the type $RM_L$ in the elongation phase of proteins. The core of the regulation mechanism, {\em the stringent response}, is that the complex $RM_L$ can use them as a sensor of the lack of amino-acids in the cell. It can detect these uncharged tRNAs and initiate  a pathway leading to the production of  a specific metabolite, (p)ppGpp, that will inhibit/block ribosomes and, therefore, the creation of complexes of the type $RM_L$. The number of uncharged tRNAs is therefore the key quantity in this regulation mechanism.  See~\citet{Bai}, \citet{Traxler} and~\citet{Zhu} for example. 

This is, of course,  a simplification of the current consensus on the stringent response but, from a modeling perspective, this seems to be a reasonable representation. More biological details and further references are provided in Chapter~4 of~\citet{JanaPhD}. It it difficult to overstate the importance of the stringent response in the general economy of bacterial cells. Many types of bacterial cells have such a regulation, with variations, but with the same basic principles. It is a key component of the properties of adaptation of these cells to various environments. Recall that these organisms exist for almost four billions years. 

This chemical reaction context with chemical species, suppliers and resources with interactions via pairing mechanisms is quite common in molecular biology. The transcription step  can be described similarly, $R$-particles are polymerases, $M$-particles are genes and $U$-particles are nucleotids such as ATP or GTP and similarly for the duplication of DNA.  See~\citet{Watson}.

\subsection{A Chemical Reaction Network}\label{Trans-Intro}
We now give a more detailed presentation of the kinetics of this system in terms of  a chemical reaction network. See~\citet{Feinberg} for the general definitions and results on chemical reaction networks.  See Section~\ref{Model-Sec} for a detailed Markovian description.

There are $10$ chemical species,
\[
{\cal S}\steq{def} \{{\cal R}, {\cal R}_S, {\cal RM}_I, {\cal RM}_L, {\cal U}, {\cal UQ},  {\cal Q}, {\cal M}, {\cal P}, \emptyset\}.
\]
and the set of reaction rates is
\begin{equation}\label{RRates}
  {\cal K}\steq{def}\{\krs,\ksr,\klr,\kq0,\k0q,\kri,\kil,\kqu\}.
\end{equation}
All these constants are assumed to be positive. The set of chemical reactions  is given by
\[
\begin{cases}
  &{\cal R}+{\cal M} \xrightharpoonup{\kri} {\cal RM}_I \xrightharpoonup{\kil} {\cal RM}_L,\quad
  \emptyset \xrightleftharpoons[\kq0]{\k0q}  {\cal Q}, \quad   {\cal U} + {\cal Q}\xrightharpoonup{\kqu}  {\cal UQ},\\
    &{\cal R}+{\cal U} \xrightharpoonup{\krs} {\cal R}_S+{\cal U},\qquad
   {\cal R}_S\xrightharpoonup{\ksr} {\cal R}, \quad
                    {\cal RM}_L + {\cal UQ}\xrightharpoonup{\klr} {\cal U} +{\cal R}+{\cal M}+{\cal P}.
\end{cases}
\]
We detail each of these reactions. 
\begin{enumerate}
\item {\sc Pairings of $R$-Particles and $M$-Particles}.\\\label{ipair}
  In the case of the regulation of the translation step, an $R$-particle is a ribosome and an $M$-particle an mRNA. 
\begin{equation}\label{Pair-eq}
{\cal R}+{\cal M} \xrightharpoonup{\kri} {\cal RM}_I \xrightharpoonup{\kil} {\cal RM}_L
\end{equation}
The $R$-particles can be
\begin{itemize}
\item {\em Free},  the species ${\cal R}$;
\item {\em Paired} with a $M$-particle in two species, ${\cal RM}_I$ or ${\cal RM}_L$.\\ The species ${\cal RM}_I$ describes the complex after the pairing of a couple of  $M$ and $R$-particles during the {\em initiation phase}. After several steps, a complex of the species ${\cal RM}_L$ is obtained.  The complex has the functional properties to start the production of a $P$-particle via a  pairing with  a complex of species ${\cal UQ}$. See item~\eqref{iprod}.
\item {\em Sequestered}, species ${\cal R}_S$.  See item~\eqref{iseq}.
\end{itemize}

\medskip
\item {\sc Resources: $Q$-Particles}.\\
For  the regulation of the translation step, a $Q$-particle is an amino-acid. The reactions are 
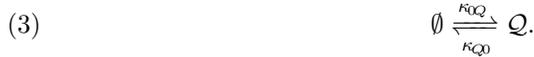
\begin{equation}\label{Ress-eq}
  \emptyset \xrightleftharpoons[\kq0]{\k0q}  {\cal Q}.
\end{equation}
They describe the creation of these particles and their possible degradation. A $Q$-particle can be also used/consumed to produce a $P$-particle. See item~\eqref{iprod}. If the production rate $\k0q$  is large, the biological environment is favorable. The interesting situation is when it is not the case, this is where the regulation mechanism  is playing its role. 

\medskip
\item {\sc Suppliers: $U$-Particles}.\\
A $U$-particle is a tRNA in  the regulation of the translation step. Such a particle can be paired with a $Q$-particle, as follows,
  \[
    {\cal U} + {\cal Q}\xrightharpoonup{\kqu}  {\cal UQ}.
    \]
A $U$-particle can thus be free, i.e. the species ${\cal U}$, or paired, the species ${\cal UQ}$. 
\item {\sc Sequestration of $R$-Particles}. \label{iseq}\\
The regulation mechanism is represented by the following reactions,
  \[
  {\cal R}+{\cal U} \xrightharpoonup{\krs} {\cal R}_S+{\cal U}\qquad
   {\cal R}_S\xrightharpoonup{\ksr} {\cal R}.
   \]
The number of free $U$-particles is the key characteristic. Because of the kinetics based on the law of mass action, see Section~\ref{Model-Sec}, if the number of free $U$-particles is large, the $R$-particles will be often sequestered, i.e. in a state where they cannot contribute to the production of $P$-particles. A sequestered $Q$-particle becomes free after some time.

\medskip
  \item {\sc Elongation Phase: Production of $P$-Particles}.\label{iprod}
\begin{equation}\label{Prod-eq}
      {\cal RM}_L + {\cal UQ}\xrightharpoonup{\klr} {\cal U} +{\cal R}+{\cal M}+{\cal P}.
\end{equation}
Roughly speaking, a complex of species ${\cal UQ}$ ``carries'' a $Q$-particle to a complex ${\cal RM}_L$ to produce of $P$-particle (protein). 
\end{enumerate}
In the context of  the regulation of the translation step, Relation~\eqref{Prod-eq} indicates that a protein is created with one amino-acid.This is a caricature of the real system of course since the average length of a protein is around 300 amino-acids for {\it E. coli}. Nevertheless, from a modeling perspective, it does make sense. We believe that the main qualitative results of this paper, the existence of three regimes, still hold when one considers $J$ different types  $Q_1$, \ldots, $Q_K$ of amino-acids and  the distribution of the sequence of amino-acids of a protein is i.i.d. distribution with a random length. The dimension of the state space is much larger in this detailed model: $1{+}4J$ for a {\em minimal} model instead of $5$ for the model of this paper, which already requires a careful analysis. 

\subsection{The Three Regimes}\label{RegSec}
We give a quick and informal picture of the results we show in this paper. Precise statements of the main theorems are in Section~\ref{Model-Sec}. 

In the analysis of the regulation of transcription, see~\citet{FRZ}, we have shown the result that there are two regimes depending on the initiation rates of rRNAs, overloaded/under-loaded. For the regulation of the translation step, the situation is, quite surprisingly, more intricate since we will show that there are in fact three asymptotic regimes. 

Initially, the total number of $R$-particles (free, paired, sequestered)  is fixed and equal to $N$, this is the scaling parameter of our analysis. The input rate of $Q$-particles is assumed to be the order of $N$. There are other (natural) scaling assumptions which are detailed in  Section~\ref{Model-Sec}.

The state of the system at time $t{>}0$ is given by a vector of $\N^5$, $$X_N(t)\steq{def} (S_N(t),R_N(t),L_N(t),Q_N(t),U_N(t)),$$ for the respective numbers of copies of  the species ${\cal R}_S$, ${\cal R}$, ${\cal RM}_L$, ${\cal Q}$, ${\cal U}$. 

\begin{enumerate}
\item {\sc Stable Network}. If the input rate of $Q$-particles is above some threshold $\kappa_0$. The variables $(S_N(t),R_N(t),L_N(t),U_N(t))$ are $O(1)$. 
\item {\sc Under-Loaded Network}. When the input rate  of $Q$-particles is below $\kappa_0$, there is a condition ${\cal C}$, see Relation~\eqref{C-Cond} of Section~\ref{Model-Sec},  depending on the parameters of the system such that,
  \begin{enumerate}
\item[i)] if ${\cal C}$ holds, then the variables $(L_N(t),R_N(t),Q_N(t))$  are $O(1)$ and the process $(S_N(t)/N,U_N(t)/N)$ converges in distribution to a non-trivial limit. 
 \item[ii)]  if ${\cal C}$ does not hold then the variables $(R_N(t),Q_N(t))$  are $O(1)$ and the process $(L_N(t)/N,S_N(t)/N)$ converge in distribution to a non-trivial limit. 
  \end{enumerate}
\end{enumerate}
In case a),  when $Q$-particles are created they can be paired with $U$-particles of the network, the complexes of the species ${\cal RM}_L$  produce $P$-particles, and disappear, rapidly. When there are not sufficiently many $U$-particles, $(U_N(t))$ is O(1) in this regime, the $Q$-particles are degraded at some fixed rate. 

The case b)-i) is more interesting. This is the unexpected regime in fact but it is also a validation of the efficiency of this regulation mechanism. 
In this situation, the input rate of  $Q$-particles is not sufficient, nevertheless the regulation mechanism, via the sequestration, manages to keep the process $(L_N(t))$ of $R$-particles in elongation of the order of $O(1)$.  It turns out that in this case the regulation mechanism  is sequestering just the right number of $R$-particles so that most of resources are used to produce $P$-particles.  It is defined as the {\em optimal sequestration} regime. See Theorem~\ref{Stab-1}. 

In the case of  the regulation of the translation step, this means that if the flow of amino-acids is low, there will be few ribosomes in elongation, thereby reducing the probability of non-functional proteins. Given the adaptation capabilities of bacterial cells, it is very likely that Condition~{\cal C} holds in practice. It should be noted that Condition~{\cal C} has some intuitive explanation but this is not entirely clear to us since it involves the numerous parameters of our model. See  Theorem~\ref{Theo-Reg}. 

In the case b)-ii), the situation is not favorable from this point of view since the  number of $R$-particles in elongation of the order of $N$ and therefore prone to the production of ``wrong'' proteins.  

In a somewhat different context, mathematical models of the regulation of gene expression has been analyzed in a deterministic context essentially. This is the literature of regulatory gene networks. See \citet{Santillan}, \citet{Mestl} and~\citet{Hecker}.  See also the general references~\citet{Mackey} and Chapter~6 of~\citet{Bressloff}, and the references therein.  

\subsection{Stochastic Averaging Principles}
An asymptotic regime is determined in fact by the subset of coordinates of $(X_N(t))$ which are $O(1)$, the other coordinates being of the order of $N$.  See Definition~\ref{Order-Def} of Section~\ref{Order-Sec}. Each asymptotic regime corresponds to  a stochastic averaging principle. The coordinates of $(X_N(t))$ which are $O(1)$ are associated to a class of fast processes and the others are the slow components.

Early works on the proof of averaging principles are due to Has'minski\v{\i}. See~\cite{Khasminskii0,Khasminskii1}. Chapter~7 of~\citet{Freidlin} considers these questions in terms of the convergence of Cesaro averages of the fast component. \citet{Papanicolaou} has introduced a  stochastic calculus approach to these problems, mainly for diffusion processes. \citet{Kurtz} has developed this approach to jump processes.

Averaging principles have already  been used in various situations to study chemical reaction networks (CRNs), \citet{BallKurtz}, \citet{Kang},  \citet{Kim2017}, \citet{LR23,LR24}, \ldots It should be noted that for most of these examples of CRNs, the fast process is generally of dimension $1$, a birth and death process. The tightness of the associated occupation measure is generally proved with standard arguments. Its invariant distribution, which is crucial to express the asymptotic behavior of the slow component has an explicit expression.

In our case, the state space is of dimension five and the state space of the fast process is in general of dimension three. The proofs of the convergence in distribution of the sequence of occupation measures and of the slow process require a careful analysis. A significant difficulty of this context is of showing the  crucial property that if the coordinates of the  initial state has the ``correct'' orders of magnitude in the scaling parameter, then this property also holds on any finite time interval. This is not only a formal complication as it will be seen: Under the conditions of  Theorem~\ref{Stab-1}, this property is not true, one has to restrict the set of initial states in the convergence result for this reason. 

In a classical stochastic averaging setting such as in~\citet{Kurtz}, \citet{BallKurtz}, \citet{Kang},  the tightness/convergence  properties of fast processes are expressed in terms of their occupation measures which is sufficient in general to derive a limiting equation for the slow processes. Here,  we need a control on the orders of magnitudes of the {\em sample paths} of fast processes and not only their occupation measure. Several technical results, Propositions~\ref{MMIHit}, \ref{Prop0}, \ref{Prop1} are used for this purpose. We believe that this type of approach may be also a useful tool to prove convergence results for CRNs when more than two timescales are involved. 

The proofs of the main convergence results are achieved by using  these technical results on fast processes and via a sequence of steps involving
\begin{enumerate}
\item introduce a stopping time to stop the slow processes and the occupation measure so that the components of $(X_N(t))$ remain in the ``correct'' region of the state space. The main difficulty is of showing that this stopping time is not vanishing to $0$ as $N$ gets large. 
\item A coupling argument with an auxiliary process is used to prove tightness properties. It is a somewhat decoupled version of the initial process.
\item For each of the asymptotic regimes, the fast process is expressed in terms of a network of several $M/M/\infty$ queues. See Section~\ref{MMISec}. The invariant distribution of these networks has in fact an explicit expression. For one of them it is not expressed by a product form formula. See Proposition~\ref{FastInv}.
\item In Section~\ref{Reg-1-Sec}  a stability analysis of an asymptotically deterministic system to extend the convergence in distribution on an arbitrary time-interval.   This is a crucial ingredient in the proofs of Theorems~\ref{Stab-1-Max} and~\ref{Stab-1-Sat}. In the context of optimal sequestration, Theorem~\ref{Stab-1-Max}, there is a constraint on the initial state precisely for this reason. It is due to the fact that even if the system starts in the ``correct'' state space (for the orders of magnitudes of the coordinates), it may leave it before returning to it later. 
\end{enumerate}
The proof of Theorem~\ref{Stab-0} details each of these steps. 

\subsection{Organization of the Paper}
Section~\ref{Model-Sec} introduces the processes investigated and the associated definitions and notations. The main first order results are stated in this section. To stress the impact of the regulation, Section~\ref{Reg-0-Sec} investigates the CRN without a sequestration state for the $R$-particles.  There are two asymptotic regimes in this case. Section~\ref{Reg-1-Sec} investigates the general system with regulation. In this case there is an additional asymptotic regime, the optimal  sequestration regime.  Theorems~\ref{Stab-1-Max} and~\ref{Stab-1-Sat} of this section are the main convergence results of the paper. In addition, in each case, as a consequence of these results,  we establish a convergence result for the total number of $P$-particles produced. 
\section{Stochastic Model}\label{Model-Sec}

\subsection{Notations and Definitions}
If $H$ is a locally compact space ${\cal C}(H)$, resp. ${\cal C}_c(H)$, denotes the set real-valued continuous functions on $H$, resp. continuous functions with compact support on $H$. The set of Radon measures on $H$ is denoted by ${\cal M}(H)$, it is endowed with the topology of weak convergence. Finally ${\cal P}(H)$ is the space of probability distributions on $H$. See~\citet{Rudin}. 

A \cadlag process $(X(t))$ is a stochastic process whose sample paths are almost surely continuous on the right with left limit $X(s{-})$ at every point $s{>}0$.

\begin{definition}
If $(Z(t))$ is a \cadlag process on $H$, the {\em occupation measure} associated to $(Z(t))$ is the random measure $\Lambda$ on $\R_+{\times}H$ defined by
\[
\croc{\Lambda,f}=\int_0^{+\infty} f\left(t,Z(t)\right)\diff t,
\]
for $f{\in}{\cal C}_c(\R_+{\times}H)$
\end{definition}
See~\citet{Dawson} for a general presentation of random measures.

\subsection{Scaling Conditions}\label{ScalSec}
For $A{\in}\{R,M,U\}$, the total number of $A$-particles, paired or free is constant, it will be assumed that the corresponding constants $R^0_N$, $M^0_N$ and $U^0_N$,  are such that, $R^0_N{=}N$ and are asymptotically proportional to $N$, 
\begin{equation}\label{Scal}
\lim_{N\to+\infty}\left(\frac{M^0_N}{N},\frac{U^0_N}{N}\right)=(C_M,C_U),
\end{equation}
with $C_M{>}1$ and $C_U{>}0$. The condition  $C_M{>}1$ reflects the fact that there are much more $M$-particles that $R$-particles. This condition has in fact a biological motivation. See Chapter~4 of~\citet{JanaPhD}. 

It will be also assumed that  the input rate of resources, the $Q$-particles,  is also of the order of $N$.

\subsection{A Markovian Representation}

The state space of our system is 
\[
{\cal S}_N=\left\{x{=}(s,r,l,q,u){\in}\N^5:s{+}r{+}l{\le}N, u{\le}U^0_N\right\}.
\]
If $x{=}(s,r,l,q,u)$ is an element of ${\cal S}_N$,
\begin{enumerate}
\item $s$, $r$, $l$  are respectively the number of $R$-particles sequestered, free or in elongation (species ${\cal RM}_L$),
\begin{itemize}
\item  $N{-}r{-}s{-}l$ is the number of  $R$-particles in initiation (species ${\cal RM}_I$).
\item  $N{-}r{-}s$ is the number of  $R$-particles paired with an $M$-particle (species ${\cal RM}_I$ or ${\cal RM}_L$),
\item  $M^0_N{-}(N{-}r{-}s)$ is the number of free  (i.e. not paired) $M$-particles. 
\end{itemize}
\item $q$ is the number of free $Q$-particles;
 \item $u$ is the number of free $U$-particles, $U^0_N{-}u$ is the number of paired $Q$-particles. 
\end{enumerate}

The kinetics of our system is Markovian, driven by the {\em law of mass action}, see~\citet{Guldberg} and~\citet{Voit2015}.  For the reactions
\[
{\cal A} \xrightharpoonup{\kappa} {\cal C},\qquad 
{\cal A}+{\cal B} \xrightharpoonup{\kappa} {\cal C},
\]
with the chemical species ${\cal A}$, ${\cal B}$, ${\cal C}$, in state $(x_A,x_B,x_C)$, the number of copies of species  ${\cal C}$ grows at rate
$\kappa x_A$  in the first reaction and $\kappa x_Ax_B$ in the second one. 

\medskip

This gives that the  $Q$-matrix  ${\cal Q}_N$ of the associated Markov process of our model
\[
(X_N(t))=(X_{N,i}(t),i{\in}\{s,r,\ell,q,u\})=(S_N(t),R_N(t),L_N(t),Q_N(t),U_N(t))
\]
is defined by, for $x{=}(s,r,l,q,u)$, 
\begin{equation}\label{QMat}
x{\longrightarrow} x{+}
\begin{cases}
{\rm Jump}& {\rm Rate}\\
 \e{Q},& \k0q N,\\
 {-}\e{Q},&\kq0 q,\\
 {-}\e{U}{-}\e{Q},& \kqu uq.
\end{cases}\quad
x{\longrightarrow} x{+}
\begin{cases}
\e{R}{-}\e{L}{+}\e{U},&\klr (U^0_N{-}u)l,\\
\e{S}{-}\e{R},&\krs ru,\\
\e{R}{-}\e{S},&\ksr s,\\
{-}\e{R},&\kri(M^0_N{-}(N{-}r{-}s))r,\\
\e{L},&\kil (N{-}r{-}s{-}l),
\end{cases}
\end{equation}
where $\e{R}$, $\e{L}$, $\e{U}$, $\e{S}$ and $\e{Q}$ are the unit vectors of $\N^5$. 
 
\subsection{Occupation Measures and Orders of Magnitude}\label{Order-Sec}
For the sake of readability, we make a slight abuse of notations for the set of indices in the following, we will use the notation $\{s,r,\ell,q,u\}$ instead of $\{1,2,3,4,5\}$ and similarly for any subset.

If $I$ is a subset  of $\{s,r,\ell,q,u\}$, the random measure $\Lambda_{N,I}$ is the occupation measure associated to the process $(X_{N,i}(t),i{\in} I)$, i.e.
the random measure  on $\R_+{\times}\N^{|I|}$ defined by,
\[
\croc{\Lambda_{N,I},f}=\int_0^{+\infty} f\left(t,(X_{N,i}(t),i{\in} I)\right)\diff t,
\]
for $f{\in}{\cal C}_c(\R_+{\times}\N^{|I|})$. In some situations, the fifth coordinate  $(X_{N,u}(t)){=}(U_N(t))$ may be replaced by $(U^0_N{-}U_N(t))$ in the definition of $\Lambda_{N,I}$. 

The scaled process on the index set $I$ is defined by
\[
\left(\overline{X}_{N,I}(t)\right) \steq{def}\left(\frac{X_{N,i}(t)}{N},i{\in} I\right).
\]

We will study the asymptotic behavior of this model for a convenient $I$ when $N$ goes to  infinity. The scaling results obtained give the order of magnitude in $N$ of the chemical species,  proportional to $N$ or $O(1)$ essentially. Mathematically, this is expressed rigorously by the following definition.

The three regimes in this paper are identified by the subset of the variables which are $O(1)$, they are ``fast variables'', the other are ``slow'' variables, of the order of $N$. The convergence is expressed as follows. 
\begin{definition}\label{Order-Def}
For $I{\subset}\{s,r,\ell,q,u\}$, the coordinates of the process $(X_N(t))$ whose index is in $I$ are said to be  $O(1)$ and the others of the order of $N$ if 
\begin{itemize}
  \item The initial condition is such that, for any $N{\ge}1$, $X_{N,I}(0){=}(X_{N,i}(0),i{\in}I)$ is a fixed element of $\N^{|I|}$ and
\[
\lim_{N\to+\infty} \left(\frac{X_{N,i}(0)}{N},i{\not\in}I\right)=(x_i,i{\not\in}I)\in(0,{+}\infty)^{5-|I|}.
\]
\item   The sequence of  variables $((\overline{X}_{N,I^c}(t)), \Lambda_{N,I})$, with $I^c{=}\{s,r,\ell,q,u\}{\setminus}I$,  is converging in distribution to 
$((x(t)), \Lambda_\infty)$, where $(x(t))$ is a non-trivial continuous process and  $\Lambda_\infty$ is a random measure on $\R_+{\times}\N^{|I|}$, such that, almost surely,  $\Lambda_{\infty}(\diff x, \N^{|I|})$ is the Lebesgue measure on $\R_+$.
\end{itemize}
\end{definition}

\subsection{Stochastic Differential Equations}\label{SDE-Sec}
The Markov process $(X_N(t))$ is represented as a \cadlag process on $\N^5$, the unique solution of a stochastic differential equation (SDE).

The probability space is such that, for any element $\kappa$ of the set of reaction rates ${\cal K}$, see Relation~\eqref{RRates}, there is a Poisson process  ${\cal P}_\kappa$ on $\R_+^2$ with intensity measure $\kappa \diff s{\otimes}\diff t$ on $\R_+^2$. The Poisson processes ${\cal P}_\kappa$, $\kappa{\in}{\cal K}$ are assumed to be independent.  See Chapter~1 of~\citet{Robert}.

The associated filtration is $({\cal F}_t)$, where, for $t{\ge}0$,  ${\cal F}_t$ is the completed  $\sigma$-field generated by set of random variables
\begin{equation}\label{SDEFilt}
\left\{{\cal P}_r(A{\times}[0,s)), r{\in}{\cal K}, s{\le}t, A{\in}{\cal B}(\R_+)\right\}. 
\end{equation}
Throughout the paper, the notions of adapted, optional processes, of martingale and stopping time are implicitly with respect to this filtration. See~\citet{Rogers2} for general definitions  and results of stochastic calculus. 

For $a{\ge}0$ and $\kappa{\in}{\cal K}$, we will use  the differential notation, 
\[
{\cal P}_\kappa((0,a),\diff t) =\int_{s=0}^a{\cal P}_\kappa(\diff s,\diff t).
\]
Finally, throughout the paper, for $a{>}0$, $\Pois{a}$ denotes the Poisson distribution on $\N$ with parameter $a$. 

Let $(X_N(t))=(S_N(t),R_N(t),L_N(t),Q_N(t),U_N(t))$ be the unique solution of the following system of SDEs 
\begin{equation*}
  \diff S_N(t)={\cal P}_{\krs}\left(\left(\rule{0mm}{4mm}0,R_N(t{-})U_N(t{-})\right),\diff t\right) 
        {-}{\cal P}_{\ksr}\left(\left(\rule{0mm}{4mm}0,S_N(t{-})\right),\diff t\right),\phantom{aaaaaaa}
\end{equation*}
\begin{multline*}
  \diff R_N(t)={\cal P}_{\klr}\left(\left(\rule{0mm}{4mm}0,L_N(t{-})(U^0_N{-}U_N(t{-}))\right),\diff t\right) 
{+}{\cal P}_{\ksr}\left(\left(\rule{0mm}{4mm}0,S_N(t{-})\right),\diff t\right)\\
{-}{\cal P}_{\krs}\left(\left(\rule{0mm}{4mm}0,R_N(t{-})U_N(t{-})\right)\diff t\right)\\
{-}{\cal P}_{\kri}\left(\left(\rule{0mm}{4mm}0,(M^0_N{-}(N{-}R_N(t{-}){-}S_N(t{-})))R_N(t{-})\right),\diff t\right),
\end{multline*}
\begin{multline*}
  \diff L_N(t)={\cal P}_{\kil}\left(\left(\rule{0mm}{4mm}0,N{-}R_N(t{-}){-}S_N(t{-}){-}L_N(t{-})\right),\diff t\right) \\
{-}{\cal P}_{\klr}\left(\left(\rule{0mm}{4mm}0,L_N(t{-})(U^0_N{-}U_N(t{-}))\right),\diff t\right),
\end{multline*}
\begin{multline*}
  \diff Q_N(t)={\cal P}_{\k0q}\left(\left(\rule{0mm}{4mm}0,N\right),\diff t\right) 
{-}{\cal P}_{\kqu}\left(\left(\rule{0mm}{4mm}0,Q_N(t{-})U_N(t{-})\right),\diff t\right)\\
{-}{\cal P}_{\kq0}\left(\left(\rule{0mm}{4mm}0,Q_N(t{-})\right),\diff t\right),
\end{multline*}
\begin{multline*}
\diff U_N(t)={\cal P}_{\klr}\left(\left(\rule{0mm}{4mm}0,L_N(t{-})(U^0_N{-}U_N(t{-}))\right),\diff t\right)\\
{-}{\cal P}_{\kqu}\left(\left(\rule{0mm}{4mm}0,Q_N(t{-})U_N(t{-})\right),\diff t\right),
\end{multline*}
with initial point $(s,r,l,q,u){\in}{\cal S}_N$. 

\subsection*{Production Process}
Since the $P$-particles are created via Reaction~\eqref{Prod-eq},  with these notations,  the number $P_N(t)$ of $P$-particles produced up to time $t{\ge}0$, is therefore given by 
\begin{equation}\label{Prott}
  P_N(t)\steq{def}
  \int_{(0,t]}{\cal P}_{\klr}\left(\left(\rule{0mm}{4mm}0,L_N(s{-})(U^0_N{-}U_N(s{-}))\right),\diff s\right).
\end{equation}

The existence and uniqueness of the solution $(X_N(t))$ of this system of SDEs is standard by constructing by induction the sequence of the instants of jump and the state at these instants. Such a solution has the same distribution as the Markov process with $Q$-matrix ${\cal Q}_N$ and initial point $(s,r,l,q,u)$. It is easily checked by showing that it is the solution of the martingale problem associated to $Q_N$. See Section~IV.20 of~\citet{Rogers2} for example.

To simplify the mathematical model, it is tempting to merge the initiation and elongation phases,  into one phase.  It would reduce the dimension of the state space to $4$ instead of $5$. As a model of the regulation of gene expression, this is of course not accurate from a biological point of view. Mathematically, it can be shown that the associated Markov process of this reduced system does not exhibit the qualitative properties described in Section~\ref{RegSec}: When the network is under-loaded, the number of complexes $RM_L$, the process $(L_N(t))$, is always of the order of $N$ so that Regime~(b)~i) does not exist.  This would imply that for a biological cell under stress a significant fraction of proteins would exhibit errors, which is unlikely biologically.

\subsection*{The Use of Poisson Processes on $\R_+^2$ in the SDEs}  If $(\lambda(t))$ is a \cadlag adapted process, a Poisson point process with intensity $(\lambda(t))$ can be represented in two ways:
\begin{enumerate}
\item Following Kurtz, see~\citet{Ethier}, if ${\cal N}$ is a Poisson point process with rate $1$ on $\R_+$, the counting  measure of such a point process can be expressed as
  \[
({\cal A}(t))\steq{def}  \left({\cal N}\left(0,\int_0^t\lambda(s)\diff s\right)\right).
  \]
\item In our paper we take the representation
  \[
({\cal B}(t))\steq{def}    \left(\int_0^t{\cal P}((0,\lambda(s-))),\diff s\right),
  \]
  where ${\cal P}$ is a Poisson process on $\R_+^2$ with intensity measure $\diff s{\otimes}\diff t$. 
\end{enumerate}
It is not difficult to see that $({\cal A}(t))$ and $({\cal B}(t))$ have the same distribution.

It should be noted that the filtration $({\cal F}_t)$ we have defined is dependent of the process $(\lambda(t))$.
 A natural filtration for $({\cal A}(t))$ would, a priori, depend on  $(\lambda(t))$.  When  coupling  constructions are considered, there may be  different such processes $(\lambda(t))$, with  a common  driving Poisson process. The definition of the filtration, which is crucial for martingale, stopping time properties, is not impossible in this case, but may be quite cumbersome to define properly. 

 \subsection{The Different Asymptotic Regimes}
To stress the impact of the regulation mechanism, we have also analyzed in Section~\ref{Reg-0-Sec} the model without the regulation, i.e. when $R$-particles cannot be sequestered. The regulation mechanism is analyzed in Section~\ref{Reg-1-Sec}.

 \subsubsection{System without Regulation}\ \\
The state space in this case is  
\[
{\cal S}_N=\left\{x{=}(r,\ell,q,u){\in}\N^4:r{+}\ell{\le}N, u{\le}U^0_N\right\},
\]
and the Markov process has four coordinates
 \[
(X_N(t))=(R_N(t),L_N(t),Q_N(t),U_N(t)). 
 \]

{\em A heuristic picture.} In view of the discussion of the introduction, a desired state is that the processes of $R$-particles in elongation $(L_N(t))$ and the process of the number of  free $U$-particles $(U_N(t))$ should be $O(1)$. Since $C_M{>}1$, the transition from ${\cal R}$ to ${\cal RM}_I$ is quick, so the same property holds for $(R_N(t))$.  Consequently, the coordinates $(R_N(t),L_N(t),U_N(t))$ should be $O(1)$. This implies that all, but a finite number, of the $N$  $R$-particles are in the species ${\cal RM}_I$. Similarly, the arriving $Q$-particles are quickly paired with a $U$-particle. The arrival rate of $Q$-particles is $\k0q$ and their consumption rate, via the production of $P$-particles is of the order of $\kil N$. Hence if $\k0q{>}\kil$, there will be an excess of $Q$-particles, so that the process   $(L_N(t))$ of the number of  proteins in elongation is $O(1)$ with high probability on finite time intervals, which is desirable to avoid the production of non-functional proteins in the case of gene expression. 

Our main results are summarized in the following theorem. A more precise version of these results are given by Propositions~\ref{Stab-0} and~\ref{Sub-0} of Section~\ref{Reg-0-Sec}.
\begin{theorem}\label{Reg-0-Th}
  Under the scaling conditions~\eqref{Scal},
 \begin{enumerate}
\item\label{item-a}  if $\k0q{>}\kil $, and the initial state is $x_N(0)=(r,\ell,q_N,u){\in}{\cal S}_N$ and
  \[
  \lim_{N\to+\infty} \frac{q_N}{N}=q_\infty{>}0,
  \]
The coordinates $(R_N(t),L_N(t),U_N(t))$ are $O(1)$ and the convergence in distribution 
  \[
  \lim_{N\to+\infty} \left(\frac{Q_N(t)}{N}\right)=(q(t)),
  \]
  holds, where $(q(t))$ is the solution of the ODE,
  \[
  \dot{q}(t)=\k0q -\kil +\kq0 q(t),
  \]
  starting at $q_\infty$. 
\item 
  If $\k0q{<}\kil$ and  $x_N(0){=}(r,\ell_N,q,U^0_N{-}u)$, with $r$, $q$, $u{\in}\N$ and
  \[
  \lim_{N\to+\infty} \frac{\ell_N}{N}=\ell_0{\in}(0,1),
  \]
  The coordinates $(R_N(t),Q_N(t),U_{N}^0{-}U_N(t))$  are $O(1)$ and the convergence in distribution
  \[
  \lim_{N\to+\infty} \left(\frac{L_N(t)}{N}\right)=(\ell(t)),
  \]
  where $(\ell(t))$ is the solution of the ODE,
\begin{equation}\label{eq-ell}
  \dot{\ell}(t)=\kil-\k0q-\kil\ell(t),  
\end{equation}
  starting at $\ell_0$.
 \end{enumerate}
 \end{theorem}
Note that the equilibrium point of the ODE~\eqref{eq-ell} is $\eta{\steq{def}}1{-}\k0q/\kil{>}0$, the fraction of $R$-particles in elongation is greater than $(\ell_0\wedge\eta)/2$ with high probability on any finite time interval. Hence when  $\k0q{<}\kil$, the system has a large number $R$-particles in elongation. In the context of gene expression, this implies that non-functional proteins will be produced on a regular basis. 
 \subsubsection{System with Regulation}

Our main results on the regulation are summarized in the following theorem. A more detailed version of these results, with the explicit expressions of the corresponding limiting occupation measures,  is given by  Propositions~\ref{Stab-1-Sat} and~\ref{Stab-1-Max}  of Section~\ref{Reg-1-Sec}.
\begin{theorem}\label{Theo-Reg}
Under the scaling conditions~\eqref{Scal},
\begin{enumerate}
\item   If $\k0q{>}\kil$, the assertions of \eqref{item-a} of Theorem~\ref{Reg-0-Th} hold and the coordinates $(S_N(t),R_N(t),L_N(t),U_N(t))$ are $O(1)$. 
\item  If $\k0q{<}\kil$ and
\begin{equation}\label{C-Cond}
\frac{\ksr}{\krs }\frac{\kri }{\k0q}\left(C_M{-}\frac{\k0q}{\kil}\right)\left(1{-}\frac{\k0q}{\kil}\right) < C_U
\end{equation}
and if  $x_N(0){=}(s_N,r,\ell,q,u_N)$, with $r$, $q$, $\ell{\in}\N$ and
  \[
  \lim_{N\to+\infty} \left({s_N}/{N},{u_N}/{N}\right)=(s_0,u_0) \in V_0.
  \]
where $V_0$ is an open subset of $(0,1){\times}(0,C_U)$, then  $(R_N(t),L_N(t),Q_N(t))$ is $O(1)$ and, for the convergence in distribution, the relation
  \[
  \lim_{N\to+\infty} \left({S_N(t)}/{N},{U_N(t)}/{N}\right)=(s(t),u(t)),
  \]
holds,  where $(s(t), u(t))$ is the solution of the ODE,
\begin{align*}  
\dot{u}(t)&{=}\kil(1{-}s(t)){-}\k0q\\
\dot{s}(t)&{=}\krs u(t)\frac{\kil(1{-}s(t)){+}\ksr s(t)}{\kri (C_M{-}(1{-}s(t))){+}\krs u(t)}{-}\ksr s(t),
\end{align*}
starting at $(s_0,u_0)$.
\item   If $\k0q{<}\kil$ and 
\begin{equation}\label{D-Cond}
\frac{\ksr}{\krs }\frac{\kri }{\k0q}\left(C_M{-}\frac{\k0q}{\kil}\right)\left(1{-}\frac{\k0q}{\kil}\right) > C_U
\end{equation}
hold, and if  $x_N(0){=}(s_N,r,\ell_N,q,U^0_N{-}u)$, with $r$, $q$, $u{\in}\N$ and
  \[
  \lim_{N\to+\infty} \left(\frac{s_N}{N},\frac{\ell_N}{N}\right)=(s_0,\ell_0) \text{ with } \ell_0, s_0{>}0 \text{ and }s_0{+}\ell_0{<}1.
  \]
then the coordinates   $(R_N(t),Q_N(t),U^0_N{-}U_N(t))$ are $O(1)$ and, for the convergence in distribution, the relation 
  \[
  \lim_{N\to+\infty} \left(\left(\frac{S_N(t)}{N},\frac{L_N(t)}{N}\right)\right)=(s(t),\ell(t)),
  \]
  holds,   where $(s(t), \ell(t))$ is the solution of the ODE,
\begin{align*}  
\dot{\ell}(t)&{=}\kil(1{-}\ell(t){-}s(t))-\k0q,\\
\dot{s}(t)&{=}\krs C_U\frac{\k0q{+}\ksr s(t)}{\kri (C_M{-}1{+}s(t){+}\ell(t))){+}\krs C_U}{-}\ksr s(t)
\end{align*}
starting at $(s_0,\ell_0)$. 
\end{enumerate}
\end{theorem}
In case (b), despite that there are not sufficiently many $Q$-particles, the number of $R$-particles in elongation remains finite in contrast with (b) of Theorem~\ref{Reg-0-Th} when there is no regulation. The $R$-particles are in someway put aside by the sequestration, thereby showing the efficiency of this scheme.  There are limitations to the sequestration mechanism nevertheless, since this property of stability does not hold anymore when Condition~\eqref{C-Cond} is not valid. 

\subsection{The $M/M/\infty$ process}\label{MMISec}
We introduce a classical birth and death process, the $M/N/\infty$ queue, technical results for this process play an important role in the proof of our convergence results.  See Chapter~6 of~\citet{Robert}. 
\begin{definition}
The ${M/M/\infty}$ queue with input rate $\lambda{\ge}0$ and service rate $\mu{>}0$ is a  Markov process  on $\N$ with transition rates, for $x{\in}\N$,
\[
x\longrightarrow
\begin{cases}
x{+}1&   \lambda,\\
x{-}1&   \mu x.
\end{cases}
\]
\end{definition}
Its invariant distribution is $\Pois{\rho}$, a Poisson distribution with parameter $\rho{=}\lambda/\mu$.
\begin{center}
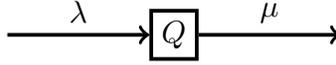
\begin{figure}[H]
\resizebox{5cm}{!}{%
\begin{tikzpicture}[->,node distance=10mm]

\node(M) at (-2,0){};
\node[black, very thick,rectangle,draw](L) at (0,0){$Q$};
\node(N) at (2,0){};

\path (M) edge [black,very thick,left,midway,above] node {$\lambda$} (L);
\path (L) edge [black,very thick,left,midway,above] node {$\mu$} (N);
\end{tikzpicture}}
\caption{$M/M/\infty$ queue}\label{FigF0}
\end{figure}
\end{center}

The following technical proposition on the $M/M/\infty$ queue  will be used repeatedly in the following.
\begin{proposition}\label{MMIHit}
\begin{enumerate}\ 
\item Let $(L(t))$ be the process of an $M/M/\infty$ queue with input rate $\lambda{>}0$ and service rate $\mu{>}0$ with $L(0){=}\ell{\in}\N$, then for any  $\delta{>}0$, $T{>}0$,
    \[
    \lim_{N\to+\infty}\P\left(\sup_{t\leq T}\frac{L(N^\delta t)}{N}\ge 1\right)=0.
    \]
  \item Let $(L_N(t))$ be the process of an $M/M/\infty$ queue with input rate $\lambda N{>}0$ and service rate $\mu{>}0$ with initial condition $L_N(0)$ such that,
    \[
    \lim_{N\to+\infty} \frac{L_N(0)}{N}=\ell_0,
    \]
    then, for any $\delta{>}0$ and $\eps{>}0$,
    \[
    \lim_{N\to+\infty}\P\left(\sup_{0\le t\leq T}\left|\frac{L_N(N^\delta t)}{N} - \rho\right|>|\ell_0{-}\rho|{+}\eps \right)=0. 
    \]
    with $\rho{=}{\lambda}/{\mu}$ and, for any $\eta{>}0$,
    \[
    \lim_{N\to+\infty}\P\left(\sup_{\eta\le t\leq T}\left|\frac{L_N(N^\delta t)}{N}{-}\rho\right|\ge \eps \right)=0. 
    \]
\end{enumerate}
\end{proposition}
In the context of the second case of this proposition, it is well known that, for the convergence in distribution
\[
\lim_{N\to+\infty}\left(\frac{L_N(t)}{N}\right)=\left(\rho{+}(\ell_0{-}\rho)e^{-\mu t}\right). 
\]
See Theorem~6.13 of~\cite{Robert} for example. The case~b) gives a further characterization.  By speeding up the timescale by a factor $N^\delta$, then the process $L_N$ is arbitrarily close to the fixed point $\rho$ outside time $0$, and do not escape from any of its neighborhoods on any finite time interval. 
\begin{proof}
Define $u_N\steq{def}\lceil N^\delta\rceil$ and 
\[
H_N=\inf\{t{>}0: L(t) \ge u_N\rceil\}.
\]
Proposition~6.10 of~\citet{Robert} shows that $(H_N/(u_N{-}1)!)$ is converging in distribution to an exponential random variable. As a consequence, the relation
\[
\lim_{N\to+\infty}\P\left(H_N\leq N^\delta T\right)=0,
\]
holds. The case~a) of the proposition is proved. 

For the case~b), a coupling is used. Assume $\rho{\le}\ell_0{<}\rho{+}\eps$, define 
$v_N^1{=}\lceil N(\rho{+}2\eps)\rceil$, and 
\[
H_N^1=\inf\{t{>}0: L_N(t) \ge v_N^1\rceil\}.
\]
We introduce $(Z_N(t))$, the Markov process with jump matrix
\[
x\longrightarrow
\begin{cases}
x{+}1&   \lambda N,\\
x{-}1&   \mu N(\rho{+}\eps)\ind{x>N(\rho{+}\eps)}.
\end{cases}
\]
It is easy to construct a coupling of $(Z_N(t))$ and $(L_N(t))$, such that $L_N(t){\le}Z_N(t)$ holds for all $t{\le}H_N^1$.

The $(Z_N(t))$ is a random walk with a reflection at $v_N^0=\lfloor N(\rho{+}\eps) \rfloor$.
This process can be expressed as $v_N^0{+}A(Nt)$, where $(A(t))$ is a birth and death process on $\N$ with birth rate $\lambda$ and death rate $\lambda{+}\mu\eps$. This is the process of the $M/M/1$ queue, see Chapter~5 of~\cite{Robert}. 

The coupling gives that $T_N{\le}N H_N^1$, with
\[
T_N=\inf\{t{\ge}0: A(t) \ge N\eps \}.
\]
Proposition~(5.11) of~\citet{Robert} shows that there exists $\alpha{\in}(0,1)$ such that the sequence $(\alpha^nT_N)$ converges in distribution to an exponential distribution. This implies that $(T_N/N^{p})$ converges in distribution to ${+}\infty$ for any $p{>}0$, hence
\[
\lim_{N\to+\infty}\P(H_N^1\leq N^\delta T)=0.
\]
The case  $\rho{\ge}\ell_0{>}\rho{-}\eps$ is handled in a similar way. The first relation of~b) is therefore established. 

For $\eps{>}0$, the convergence of $(L_N(t)/N)$ to $(\rho{+}(\ell_0{-}\rho)e^{-\mu t})$, shows that if 
\[
H_N^2\steq{def}\inf\left\{t{\ge}0: \frac{L_N(t)}{N}\in\left(\rho{-}\frac{\eps}{2},\rho{+}\frac{\eps}{2}\right)\right\},
\]
then there exists $K{>}0$ such that $\P(H_N^2{\ge}K)$ is arbitrarily small uniformly in $N$. In particular $H_N^2$ is less than $\eta N^\delta$ with high probability. Our last assertion is proved by using the strong Markov property of $(L_N(t))$ for the stopping time $H_N^2$ and the first relation of~b). 
\end{proof}

\section{A Model without Regulation}\label{Reg-0-Sec}
In this section we investigate the production of $P$-particles when there is no possibility of sequestering $R$-particles. There are two regimes described roughly as follows. 
\begin{enumerate}
\item When $N$ is large, if the input rate of resources, $Q$-particles, is above some threshold then all $R$-particles but a finite number are paired in the initiation phase, the chemical species ${\cal RM}_I$. In particular there is a finite number of  paired $R$-particles waiting for a $Q$ particle, the chemical species ${\cal RM}_L$.
\item Otherwise, if the input rate is strictly less than this threshold the number of  paired $R$-particles waiting for a $Q$ particle is of the order of $N$.
\end{enumerate}

\begin{center}
\begin{figure}[H]
\resizebox{12cm}{!}{%
\begin{tikzpicture}[->,node distance=8mm]

\node[black, very thick,circle,draw](R_F) at (1,0){${\cal R}$};
\node[black, very thick, circle,draw](R_M) at (6.5,0){${\cal RM}_I$};
\node[black, very thick, circle,draw](R_1) at (10,0){${\cal RM}_L$};
\node[black, very thick, circle,draw](P) at (12,-1){$\,{\cal P}_r\,$};

    \node () [black,very thick,rectangle,draw]at (1,1){{\bf Free }};
    \node () [black]at (-1,1){{\sc $R$-Particles}};
    \node () [black,very thick,rectangle,draw]at (6.3,1){{\bf Initiation }};
    \node () [black,very thick,rectangle,draw]at (9.9,1){{\bf Elongation }};

  \path (R_F) edge [black,very thick,left,midway,above] node {$ \kri(M^0_N{-}(N{-}R))R$} (R_M);
  \path (R_M) edge [black,very thick,left,midway,above] node {$\kil (N{-}L{-}R) $} (R_1);
   \path (R_1) edge [black,very thick,bend left=45,midway,above] node {$\klr L(U^0_N{-}U)$} (R_F);
   \path (R_1) edge [black,very thick,bend left=-110,below] node {} (P);
   
  \node[black, very thick, circle,draw](Q1) at (1,-4.5){${\cal Q}$};
  \node[black, very thick, circle,draw](0) at (5,-4.5){$\emptyset$};

  \path (Q1) edge [black,very thick,bend left=25,midway,above] node {$\kq0 Q$} (0);
  \path (0) edge [black,very thick,bend left=25,midway,below] node {$\k0q N $} (Q1);
  \node () [black]at (-1,-3.5){{\sc Resources }};

  \node[black, very thick, circle,draw](T1) at (8,-4.5){${\cal U}$};
  \node[black, very thick, circle,draw](0) at (14,-4.5){${\cal UQ}$};

  \path (T1) edge [black,very thick,bend left=25,midway,above] node {$\kqu Q U$} (0);
  \path (0) edge [black,very thick,bend left=25,midway,below] node {$ \klr L(U^0_N{-}U)$} (T1);

  \node () [black]at (7,-3.5){{\sc Transoporters }};
\end{tikzpicture}}
\caption{Network without Regulation}
\end{figure}
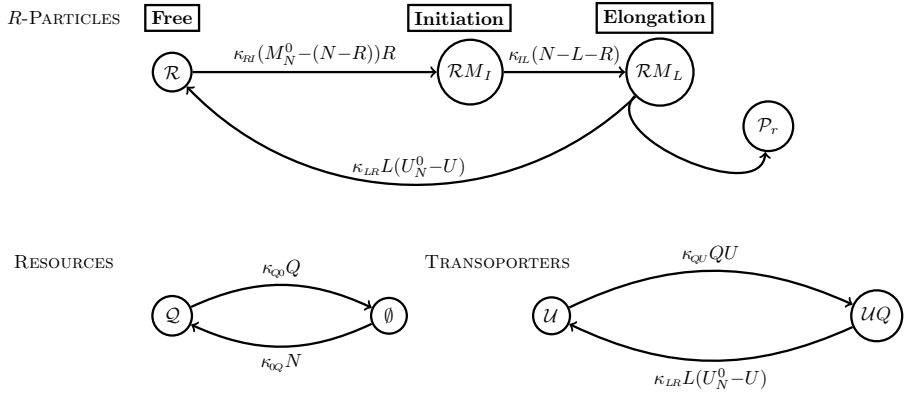
\end{center}

Without the sequestration mechanism, the state space is 
\[
{\cal S}_N=\left\{x{=}(r,\ell,q,u){\in}\N^4:r{+}l{\le}N, u{\le}U^0_N\right\},
\]
and the Markov process is  $(X_N(t)){=}(R_N(t),L_N(t),Q_N(t),U_N(t))$, with $Q$-matrix, for $x{=}(r,\ell,q,u){\in}{\cal S}_N$, 
\[
x{\longrightarrow} x{+}
\begin{cases}
\e{R}{-}\e{L}{+}e_u,&\klr  (U^0_N{-}u)\ell,\\
{-}\e{R},&\kri (M^0_N{-}(N{-}r))r,\\
\e{L},&\kil (N{-}e{-}r),
\end{cases}\text{ and }
x{\longrightarrow} x{+}
\begin{cases}
 \e{Q},& \k0q N,\\
 {-}\e{Q},&\kq0 q,\\
 {-}\e{U}{-}\e{Q},& \kqu uq.
\end{cases}
\]
In this case the corresponding SDEs, see Section~\ref{SDE-Sec}, are given by 
\begin{equation}\label{SDEwo}
\begin{cases}
  \diff R_N(t)&={\cal P}_{\klr}\left(\left(\rule{0mm}{4mm}0,L_N(t{-})(U^0_N{-}U_N(t{-}))\right),\diff t\right)\\
 &\hspace{2cm}{-}{\cal P}_{\kri}\left(\left(\rule{0mm}{4mm}0,(M^0_N{-}(N{-}R_N(t{-})))R_N(t{-})\right),\diff t\right),\\
  \diff L_N(t)&={\cal P}_{\kil}\left(\left(\rule{0mm}{4mm}0,N{-}R_N(t{-}){-}L_N(t{-})\right),\diff t\right) \\
&\hspace{2cm}{-}{\cal P}_{\klr}\left(\left(\rule{0mm}{4mm}0,L_N(t{-})(U^0_N{-}U_N(t{-}))\right),\diff t\right),\\
  \diff Q_N(t)&={\cal P}_{\k0q}\left(\left(\rule{0mm}{4mm}0,N\right),\diff t\right) \\
 &\hspace{5mm} {-}{\cal P}_{\kqu}\left(\left(\rule{0mm}{4mm}0,Q_N(t{-})U_N(t{-})\right),\diff t\right)
{-}{\cal P}_{\kq0}\left(\left(\rule{0mm}{4mm}0,Q_N(t{-})\right),\diff t\right),\\
\diff U_N(t)&={\cal P}_{\klr}\left(\left(\rule{0mm}{4mm}0,L_N(t{-})(U^0_N{-}U_N(t{-}))\right),\diff t\right)\\
 &\hspace{2cm}{-}{\cal P}_{\kqu}\left(\left(\rule{0mm}{4mm}0,Q_N(t{-})U_N(t{-})\right),\diff t\right),
\end{cases}
\end{equation}
The following proposition is a technical result on fast processes in a simple setting. It will used in the proof of the main result of this section. Its proof uses essentially standard arguments. It is nevertheless detailed since they will be used repeatedly without justification in subsequent proofs of this paper.
\begin{proposition}\label{Prop0}
For $0{<}\alpha{\le}\beta$, let $(A_N(t),B_N(t))$ be the solution of the SDE,
\[
  \begin{cases}
  \diff A_N(t) &= {\cal P}_{\kil}\left(\rule{0mm}{4mm}(0,N),\diff t\right) 
{-}{\cal P}_{\klr}\left(\left(\rule{0mm}{4mm}0,\alpha N A_N(t{-})\right),\diff t\right),\\
  \diff B_N(t) &= {\cal P}_{\klr}\left(\left(\rule{0mm}{4mm}0,\beta N A_N(t{-})\right),\diff t\right) {-}{\cal P}_{\kqu}\left(\left(\rule{0mm}{4mm}0,\eta N B_N(t{-})\right),\diff t\right),
  \end{cases}
  \]
  with initial condition $(a,b)$.

  If $\mu_N$ is the occupation measure of $(A_N(t),B_N(t))$ on $\N^2$, then the sequence $(\mu_N)$ is converging in distribution to $\diff t{\otimes}\pi$ with 
  \[
 \pi\steq{def}\Pois{\frac{\kil}{\alpha\klr}}{\otimes}\Pois{\frac{\beta\kil}{\alpha\kqu}},
  \]
and 
  \[
  \lim_{N\to+\infty}\croc{\mu_N,f{\otimes}h}=\croc{\mu_\infty,f{\otimes}h},
  \]
  for any  function $f{\in}{\cal C}_c(\R_+)$ and $h{\in}L_1(\pi)$, in particular for any $h$ such that
  \[
|h(a,b)|\le C_1{+}C_2(a{+}b), \forall (a,b){\in}\N^2,
\]
for some constants $C_1$ and $C_2$.

For any $\delta{\in}(0,1)$, $p{\ge}1$, $T{>}0$ and $\eps{>}0$,
    \begin{equation}\label{Borne}
    \lim_{N\to+\infty}\P\left(\sup_{t\leq T}\frac{(A_N{+}B_N)(N^pt)}{N^\delta}\ge \eps\right)=0.
    \end{equation}

    Finally, the relation
    \begin{equation}\label{ErT}
      \lim_{N\to+\infty} \left(\frac{1}{N}\int_0^t A_N(u)\diff u,\frac{1}{N}\int_0^t B_N(u)\diff u \right)=\left(\frac{\kil}{\klr\beta}t,\frac{\kil\beta}{\kqu\alpha}t\right)
    \end{equation}
    holds for the convergence in distribution of processes. 
\end{proposition}
Recall that, Section~\ref{SDE-Sec},  for $a{>}0$, $\Pois{a}$ is the Poisson distribution on $\N$ with parameter $a$. In Relation~\eqref{ErT}, the convergence in distribution of the marginals is a direct consequence of the convergence of $(\mu_N)$. 
\begin{proof}
Let $(Y(t)){=}(A(t),B(t))$ be the Markov process on $\N^2$ with $Q$-matrix given by 
  \[
x{\longrightarrow} x{+}
\begin{cases}
e_{1},&\kil,\\
 e_{2}{-}e_{1},&\klr \alpha  x_1,\\
 e_{2},&\klr (\beta-\alpha)x_1,\\
{-}e_{2},&\kqu \eta x_2,
\end{cases}
\]
and initial point $(a,b)$. It is clear that $(A_N(t),B_N(t)){\steq{dist}}(Y(Nt))$. Note that the the first coordinate $(A(t))$ is an $M/M/\infty$ process, its invariant measure is a Poisson distribution with parameter $\kil/(\alpha\klr)$.  The process $(Y(t))$ itself is a  positive recurrent Markov process. 
 Indeed, it is easily checked that the function $F$ on $\N^2$ defined by $F(r,l){=}a_1x_1{+}a_2x_2$ is a Lyapounov function satisfying the assumptions of Proposition~8.14 of~\citet{Robert} provided that  the positive constants $a_1$ and $a_2$ are chosen so that $\beta a_2{<}\alpha a_1$ holds. We denote by $\pi$ its invariant distribution. 

By representing  $(Y(t))$ as the solution of an SDE, one gets  that if $Y(0)\steq{dist}(0,0)$ and $K{>}0$, then  one gets the relation
\[
\eta\kqu \E\left(\int_0^t \left(B(s){\wedge}K\right)\diff s \right)\le \beta\klr\E\left(\int_0^t A(s)\diff s\right),
\]
and by dividing this relation by $t$ and by letting $t$ go to infinity, the mean ergodic theorem for positive recurrent Markov processes gives the relation
\[
\eta\kqu \int_{\N^2} \left(b{\wedge} K\right)\pi(\diff a,\diff b)\le \beta\klr \int_{\N^2} a\pi(\diff a,\diff b)=\frac{\beta\kil}{\alpha},
\]
for all $K{>}0$. By letting $K$ go to infinity, one gets that the first moments of  $\pi$ are finite. 

For $\eps{>}0$, there exists $K_0{>}0$ such that $\pi([0,K_0]^2){\ge}1{-}\eps$, so that, for $T{>}0$, 
\[
\croc{\mu_N,[0,T]{\times}[0,K_0]^2}=\int_0^T\ind{Y(Ns){\in}[0,K_0]^2}\diff s=\frac{1}{N}\int_0^{NT} \ind{Y(Ns){\in}[0,K_0]^2}\diff s,
\]
and therefore, again by the mean ergodic theorem, 
\[
\lim_{N\to+\infty}\E\left(\croc{\mu_N,[0,T]{\times}[0,K_0]^2}\right)=T\pi([0,K_0]^2.
\]
From Lemma~1.3 of~\citet{Kurtz}, we obtain that $(\mu_N)$ is tight for the convergence in distribution.
If $h{:}\N^2{\to}\R$ is in $L_1(\pi)$ then,
\[
\croc{\mu_N,\mathbbm{1}_{[0,T]}{\otimes}h}=\frac{1}{N}\int_0^{NT} h(Y(s))\diff s\stackrel{N{\to}{+}\infty}{\longrightarrow}T\int h(y)\pi(\diff y),
\]
almost surely,  by the pointwise ergodic theorem.  We get that $(\mu_N)$ is converging in distribution to $\diff t{\otimes}\pi$ and that the convergence holds for all functions $f{\otimes}h$ such that,  $f{\in}{\cal C}_c(\R_+)$ and $h{\in}L_1(\pi)$.

Fix  $\delta_0{<}\delta$, from Lemma~\eqref{MMIHit},   for any $\eps{>}0$, and $T{>}0$, we have $\P({\cal E}_N){\ge}1{-}\eps$, for $N$ sufficiently large, with 
\[
{\cal E}_N\steq{def} \left\{ \sup_{t\leq T}\frac{A_N(t)}{N^{\delta_0}}\le 1\right\},
\]
if $(F_N(t))$ is the solution of the SDE
\[
  \diff F_N(t) = {\cal P}_{\klr}\left(\left(\rule{0mm}{4mm}0,\beta N^{\delta_0}\right),\diff t\right) {-}{\cal P}_{\kqu}\left(\left(\rule{0mm}{4mm}0,\eta F_N(t{-})\right),\diff t\right),
  \]
  with $F_N(0){=}B_N(0)$. It is easily checked that the relation $B_N(t){\le}F_N(Nt)$, for all $t{\le}T$ on the event ${\cal E}_N$.
  From Proposition~\ref{MMIHit}, there exists $K_0{>}0$ such that
  \[
  \lim_{N\to+\infty}\P\left(\sup_{t\leq T}\frac{B_N(N^p t)}{N^{\delta_0}}\geq K\right)=0.
  \]
For the last assertion of the proposition, since the possible limit is determined by the ergodic theorem,  it is enough to prove the convergence of each of the coordinates of the sequence of processes and therefore the tightness of this sequence. From the SDE for $(A_N(t))$, we have 
\[
\frac{A_N(t)}{N}=\frac{a}{N}+M_N^A(t){+}\kil t{-}\beta\klr\frac{1}{N}  \int_0^t A_N(u)\diff u,
\]
where $(M_N^A(t))$ is a martingale. With Doob's Inequality, it is easily seen $(M_N(t))$ is converging in distribution to $(0)$. We have shown, Relation~\eqref{Borne}, that $(A_N(t)/N)$ is  converging to $(0)$ and, hence, the convergence of the first coordinate of Relation~\eqref{ErT}. Using again the SDE, one gets
\[
\beta\frac{A_N(t)}{N}+\alpha\frac{A_N(t)}{N}=\beta a{+}\alpha b{+}M_N^B(t){+}\alpha\kil t-\kqu\int_0^t B_N(u)\diff u,
\]
the desired convergence follows in the same way. The proposition is proved. 
\end{proof}

To prove that the process $(Y(t))$ is a  positive recurrent Markov process, we have used a Lyapounov criterion. In fact  it is easily verified that
\[
\pi= \Pois{\frac{\kil}{\alpha\klr}}{\otimes}\Pois{\frac{\beta\kil}{\alpha\kqu}}
\]
is its invariant distribution. We could have used this explicit representation to simplify our proof. It is unfortunately not always available for all the cases studied in the following. The above method with Lyapounov function and ergodic theorems applies in these cases. 

We now prove another technical result on the perturbation of a scaled $M/M/\infty$ queue. 
\begin{proposition}\label{Prop1}
Let $(V_N(t))$ and $(W_N(t))$ be non-negative optional processes with values in $[0,N]$, resp. in $[0,U_N]$, we denote by $(A_N(t))$ the solution of
\begin{multline}\label{SDELV}
  \diff A_N(t)={\cal P}_{\kil}\left(\left(\rule{0mm}{4mm}0,N{-}V_N(t{-})\right),\diff t\right) \\
{-}{\cal P}_{\klr}\left(\left(\rule{0mm}{4mm}0,A_N(t{-})(U^0_N{-}W_N(t{-}))\right),\diff t\right),
\end{multline}
with the initial condition $A_N(0){=}a{\in}\N$

If  there exists $\delta{\in}(0,1)$ such that, for any $T{>}0$,
  \[
  \lim_{N\to+\infty}\P\left(\sup_{t{\le} T}\frac{V_N(t)}{N^\delta}\ge 1, \sup_{t{\le} T}\frac{W_N(t)}{N^\delta}\ge 1\right)=0,
  \]
then the sequence of occupation measures of $(A_N(t))$ converges in distribution to the measure $\diff t{\otimes}\Pois{{\kil}/{(\klr C_U)}}$.
\end{proposition}
\begin{proof}
In the same way as in the proof of Proposition~\ref{Prop0}, with Lemma~1.3 of~\citet{Kurtz} we can prove that the sequence $(\mu_N)$ of occupation measures of $(A_N(t))$ is tight. Lemma~1.4 of this reference shows that any limiting point $\mu_\infty$ can be represented as, for $g{\in}{\cal C}_c(\R_+{\times}\N)$, 
\[
\croc{\mu_\infty,g}=\int_0^{+\infty} \int_{\N}g(s,a)\pi_s(\diff a) \diff s,
\]
almost surely, where $(\pi_s)$ is an optional process with values in ${\cal P}(\N)$, and the convergence to $\mu_\infty$ holds for function $g$ of the type $f{\otimes}h$, for $f{\in}{\cal C}_c(\R_+)$ and $h{:}\N{\to}\R_+$, such that $(h(x)/x, x{\ge}1)$ is bounded .

  If $f:\N{\to}\R_+$ has finite support then, almost surely,  for all $t{\ge}0$,  
\begin{multline}\label{eqr2}
  f(A_N(t))=  f(a)+M_N^f(t)+\kil\int_0^t \nabla_1(f)(A_N(s))(N{-}V_N(s))\diff s
\\  +\klr\int_0^t \nabla_{-1}(f)(A_N(s))A_N(s)(U^0_N{-}W_N(s))\diff s,
\end{multline}
where $(M_N^f(t))$ is a martingale. Using the assumption on $(V_N(t))$ and $(W_N(t))$, as in the proof of Proposition~\ref{Prop0}, we obtain that the sequences of processes
\[
\left(\frac{1}{N}M_N^f(t)\right) \text{ and }\left(\frac{1}{N}\int_0^t (V_N(s){+}A_N(s)W_N(s))\diff s\right)
\]
are converging in distribution to $(0)$. By dividing Relation~\eqref{eqr2} by $N$ and by taking a convenient subsequence we obtain that the relation
\[
\left(\int_0^t \int_{\N} \left(\kil\nabla_1(f)(a)){+}\klr C_U a \nabla_{-1}(f)(a)\right)\pi_s(\diff a) \diff s\right)=(0),
\]
holds almost surely. Since the operator $\Omega$ defined by
\[
\Omega(f)(a){=}\kil\nabla_1(f)(a)){+}\klr C_U a \nabla_{-1}(f)(a)
\]
is the infinitesimal generator of the $M/M/\infty$ queue with input rate $\kil$ and service rate $\klr C_U$, by following the standard method of~\citet{Kurtz} we get that,  for $g{\in}{\cal C}_c(\R_+{\times}\N)$, the relation
\[
\croc{\mu_{\infty},g}=\int_0^{+\infty} \int_{\N}g(s,a)\Pois{\frac{\kil}{\klr C_U}}(\diff a) \diff s
\]
holds almost surely. The proposition is proved. 
\end{proof}
The Markov process of Proposition~\ref{Prop0} is related to a fast process of one of the models considered in this paper. As usual the invariant distribution  of the fast processes plays an important role in the dynamical behavior of the slow variables. We have seen that it has a product form which is a desirable property. The three-dimensional fast process  of the stability result of this section, Theorem~\ref{Stab-0}  does not have this property. Nevertheless its invariant distribution has an explicit representation. This is the result of the  following lemma. 
\begin{proposition}\label{FastInv}
For $\lambda{>}0$, $\mu_i{>}0$, $i{=}R$, $L$, $U$,  the invariant distribution of the Markov process on $\N^3$ with $Q$-matrix given by
 \[
x{\longrightarrow} x{+}
\begin{cases}
\e{R}{+}\e{U}{-}\e{L}, &\mu_{L} x_{L},\\
{-}e_R,&\mu_R x_R.\\
{-}e_U,&\mu_U x_U,\\
e_L,& \lambda,
\end{cases}
\]
is the distribution of the random vector $(X+Y_1,Z,X+Y_2)$, where $X$, $Y_1$, $Y_2$, $Z$,  are independent Poisson random variables with respective parameters
\[
\frac{\lambda}{\mu_R{+}\mu_U},\quad  \frac{\lambda\mu_U}{\mu_R(\mu_R{+}\mu_U)},\quad \frac{\lambda\mu_R}{\mu_U(\mu_R{+}\mu_U)},\quad \frac{\lambda}{\mu_L}. 
\]
\end{proposition}
\begin{proof}
This Markov process can also be represented as follows: there Poisson arrivals of particles at node $2$ with rate $\lambda$, each particle leaves the node after an exponentially distributed amount of time  with rate $\mu_2$ and, at this instant, creates a particle at node $1$ and a particle at node $3$. The distribution of the lifetimes of particles at node $1$ and $3$ are exponential with respective parameters $\mu_1$ and $\mu_3$.

We use the method  of~\citet{FLR1}, see also~\citet{Rob2019}. The trick is of using a Poisson marked point process on $\R{\times}\N^3$, 
  \[
  {\cal N}={\cal N}(\diff t, \diff x)= \sum_{n{\in}\Z}\delta_{(t_n,L_{1}^{N},L_{2}^{N},L_{3}^{N})},
  \]
  where $(t_n)$ is the sequence of the points of a Poisson process on $\R_+$ with rate $\lambda$ and for $i{=}1$, $2$, $3$, $(L_{i}^N)$ is an i.i.d. sequence with an exponential distribution with parameter $\mu_i$. All theses sequences are independent. The intensity of the Poisson process is the positive measure
  \[
  m_{\cal N}(\diff u,\diff x)\steq{def}  \lambda \diff u {\otimes}\mu_1e^{-\mu_1x_1}\diff x_1{\otimes}\mu_2e^{-\mu_2x_2}\diff x_2{\otimes}\mu_3e^{-\mu_3x_3}\diff x_3
  \]
  on $\R{\times}\R_+^3$, $F$ is a non-negative Borelian function on this state space,
  \begin{multline}\label{Lap}
\E\left(\exp\left(-\int F(u,x){\cal N}(\diff u, \diff x)\right)\right)\\=\left(\exp\left(-\int\left(1{-}e^{-F(u,x)}  m_{\cal N}(\diff u,\diff x)\right)\right)\right)
  \end{multline}
  
 If the system starts empty at time $0$, for $t{\ge}$ then it is easy to see that the  state $A(t)$ of the Markov process  at time $t$ as the same distribution as $(A_1(t),A_2(t),A_3(t)$ with
\begin{multline*}
  A_2(t)=\sum_{n{\ge}0}\ind{t_n{\le}t{<} t_n{+}L_{2}^{n}}  =\int_{[0,t]{\times}\N^3}\ind{u{\le}t{<}u{+}x_2}{\cal N}(\diff u,\diff x)\\
  \steq{dist}\int_{[-t,0]{\times}\N^3}\ind{u{\le}0{<}u{+}x_2}{\cal N}(\diff u,\diff x)\stackrel{t\to{+}\infty}{\longrightarrow}
  \int_{[-\infty,0]{\times}\N^3}\ind{u{\le}0{<}u{+}x_2}{\cal N}(\diff u,\diff x)\\ \steq{dist}  \int_{[0,+\infty]{\times}\N^3}\ind{u{\le}x_2}{\cal N}(\diff u,\diff x),
\end{multline*}
the two relations in distribution of this relation are a consequence of the invariance of the distribution of ${\cal N}$ by the two transformations of $\R{\times}\N^2$, $(s,x){\mapsto}(s-t,x)$ and $(s,x){\mapsto}(-s,x)$. See Chapter~1 of~\citet{Robert}.

Similarly for $i{=}1$, $3$
\begin{multline*}
  A_i(t)=\int_{[0,t]{\times}\N^3}\ind{u{+}x_2{\le}t{<}u{+}x_2{+}x_i}{\cal N}(\diff u,\diff x)\\
  \steq{dist}\int_{[-t,0]{\times}\N^3}\hspace{-1cm}\ind{u{+}x_2{\le}0{<}u{+}x_2{+}x_i}{\cal N}(\diff u,\diff x)
  \stackrel{t\to{+}\infty}{\longrightarrow}  \int_{[-\infty,0]{\times}\N^3}\hspace{-1cm}\ind{u{+}x_2{\le}0{<}u{+}x_2{+}x_i}{\cal N}(\diff u,\diff x)\\
\steq{dist} \int_{[0,{+}\infty]{\times}\N^3}\ind{x_2{\le}u{<}x_2{+}x_i}{\cal N}(\diff u,\diff x).
\end{multline*}

We have therefore that $(A(t))$ converges in distribution to $(A_1(\infty),A_2(\infty),A_3(\infty))$, with, for $i{\in}\{1,2,3\}$ and $j{\in}\{2,3\}$,
\[
A_i(\infty)=\int f_i(u,x){\cal N}(\diff u, \diff x), \text{ with } f_2(u,x){=}\ind{0\le u{<}x_2},  f_{j}(u,x){=}\ind{x_2{\le}u{<}x_2{+}x_j}.
\]
Since, for $j{\in}\{2,3\}$, $f_2f_j{\equiv}0$, the independence properties of Poisson properties imply that $A_2(\infty)$ is independent of $(A_1(\infty),A_3(\infty))$. Since $(A_2(t))$ has the same distribution as an $M/M/\infty$ process with input rate $\lambda$ and service rate $\mu_2$, its invariant distribution, the law of $A_2(\infty)$,  is Poisson with parameter $\lambda/\mu_2$.

The Laplace transform of $(A_1(\infty),A_3(\infty))$ at $(\xi_1,\xi_3)$ is
\begin{multline*}
  \E\left(e^{-\xi_1A_1(\infty)-\xi_3A_3(\infty)}\right)\\
  =\E\left(\exp\left(-\int_{\R_+{\times}\N^3}(\xi_1f_1(u,x)+\xi_3f_3(u,x)){\cal N}(\diff u, \diff x)\right)\right).
\end{multline*}
We can now use Relation~\eqref{Lap} to get that
\begin{multline*}
\E\left(e^{-\xi_1A_1(\infty)-\xi_3A_3(\infty)}\right)=
  \exp\left(\left(1-e^{-\xi_1-\xi_3}\right)\frac{\lambda}{\mu_1{+}\mu_3}\right.\\\left. {+}\left(1-e^{-\xi_1}\right)\frac{\lambda\mu_3}{\mu_1(\mu_1{+}\mu_3)}{+}\left(1-e^{-\xi_3}\right)\frac{\lambda\mu_1}{\mu_3(\mu_1{+}\mu_3)}\right).
\end{multline*}
The proposition is proved.
\end{proof}

We can now establish our main stability result when there is no regulation. 
\begin{theorem}[Stability]\label{Stab-0}
Under the scaling assumptions~\eqref{Scal},  if $\k0q{>}\kil $, and the initial state is $x_N(0)=(r,\ell,q_N,u){\in}{\cal S}_N$, such that
  \[
  \lim_{N\to+\infty} \frac{q_N}{N}=q_0{>}0,
  \]
  and $\Lambda_N{\in}{\cal M}(\N^3)$ is the occupation measure of $(R_N(t),L_N(t),U_N(t))$,  then, for the convergence in distribution,
  \[
  \lim_{N\to+\infty} \left(\left(\frac{Q_N(t)}{N}\right),\Lambda_N\right)=((q(t)),\Lambda_\infty),
  \]
  where $(q(t))$ is the solution of the ODE,
  \[
  \dot{q}(t){=}\k0q {-}\kil {-}\kq0 q(t),
  \]
  starting at $q_0$, and $\Lambda_\infty$ is the measure on $\R_+{\times}\N^3$ defined by, for  $F{\in}{\cal C}_c(\R_+{\times}\N^3)$,
\begin{multline}\label{eqr0}
  \croc{\Lambda_\infty,F}=\int_0^{+\infty} \int_{\N^4} F(t,a{+}v,b,a{+}w)  
\Pois{\frac{\kil }{C_\kappa}} (\diff a)\\
  \times      
\Pois{\frac{\kil }{\klr C_U}} (\diff b)
\Pois{\frac{\kil \kq0 q(t)}{\kri (C_M{-}1)C_\kappa}} (\diff v)
\Pois{\frac{\kil\kri (C_M{-}1) }{\kq0 q(t)C_\kappa}} (\diff w)
\diff t,
\end{multline}
with $C_\kappa{=}\kri(C_M{-}1){+}\kq0$. 
\end{theorem}
A glance at the SDEs~\eqref{SDEwo} shows that, provided that the process stays in the subset of points similar to the initial point,  the three process $(R_N(t)$, $(L_N(t)$, $(L_N(t)$ are ``fast processes'' in the sense that their transition rates are of the order of $N$, and therefore with large fluctuations. The kinetics of $(Q_N(t)/N)$ are $O(1))$, it will be referred to as a ``slow variable''. The main challenge of the proof of the convergence of the theorem, is of controlling the fluctuations of the fast processes so that the process stays indeed in the subset of points similar to the initial point. 
\begin{proof}
 We first give a sketch of the strategy of the proof.   We first introduce a stopping time $\tau_N$ before that either the fast process $(U_N(t)/N)$ is ``too large'' or the slow process $(Q_N(t)/N)$ is ``too small''. On the time interval $[0,\tau_N]$ an auxiliary process $(\widetilde{X}_N(t)){=}(\widetilde{R}_N(t),\widetilde{L}_N(t),\widetilde{Q}_N(t),\widetilde{U}_N(t))$ is constructed with the following  property: it has three fast processes $(\widetilde{R}_N(t))$, $(\widetilde{L}_N(t))$, $(\widetilde{U}_N(t))$ are respectively greater than our fast processes $(R_N(t))$, $(L_N(t))$, $(U_N(t))$ and the slow variable $(\widetilde{Q}_N(t)/N)$ is smaller that $(Q_N(t)/N)$. A convergence result is established for $(\widetilde{X}_N(t))$ that will show that there exists $t_0{>}0$ such that $\tau_N{\ge}t_0$ with high probability and also a more precise control of the fast processes:   the sequence of processes $(R_N(t)/\sqrt{N})$, $(L_N(t)/\sqrt{N})$, $(U_N(t)/\sqrt{N})$ are bounded in distribution on $[0,t_0]$. From there, with the help of Propositions~\ref{Prop0} and~\ref{Prop1} and with tightness arguments for occupation measures and $(Q_N(t)/N)$, one obtains the desired convergence result on $[0,t_0]$. With the condition $\k0q{>}\kil $, the procedure can repeated on the time interval $[t_0,2t_0]$, and therefore on the whole half-line.

 We fix $T{>}0$, the time interval for the convergence in distribution is $[0,T]$.

 \bigskip
 
 \noindent
{\sc Step 1: Coupling.}\\
Let $\eta{\in}(0,(q_0{\wedge}1)/4)$ and define 
\[
\tau_N\steq{def}\inf\{t>0: Q_N(t)\le \eta N \text{ or } U_N(t)\ge \eta N \}.
\]
Let $(\widetilde{X}_N(t)){=}(\widetilde{R}_N(t),\widetilde{L}_N(t),\widetilde{Q}_N(t),\widetilde{U}_N(t))$ be the solution of the SDE,
\[
\begin{cases}
\diff \widetilde{R}_N(t)&={\cal P}_{\klr}\left(\left(\rule{0mm}{4mm}0,\overline{c}_uN\widetilde{L}_N(t{-})\right),\diff t\right){-}{\cal P}_{\kri}\left(\left(\rule{0mm}{4mm}0,(\underline{c}_M{-}1)N\widetilde{R}_N(t{-})\right),\diff t\right),\\
\diff \widetilde{L}_N(t)&={\cal P}_{\kil}\left(\rule{0mm}{4mm}(0,N),\diff t\right) 
{-}{\cal P}_{\klr}\left(\left(\rule{0mm}{4mm}0,(\underline{c}_u-\eta)N\widetilde{L}_N(t{-})\right),\diff t\right),\\
\diff \widetilde{Q}_N(t)&={\cal P}_{\k0q}\left(\left(\rule{0mm}{4mm}0,N\right),\diff t\right) 
{-}{\cal P}_{\kqu}\left(\left(\rule{0mm}{4mm}0,\widetilde{Q}_N(t{-})\widetilde{U}_N(t{-})\right),\diff t\right)\\
& \qquad {-}{\cal P}_{\kq0}\left(\left(\rule{0mm}{4mm}0,\widetilde{Q}_N(t{-})\right),\diff t\right),\\
\diff \widetilde{U}_N(t)
&={\cal P}_{\klr}\left(\left(\rule{0mm}{4mm}0,\overline{c}_uN\widetilde{L}_N(t{-})\right),\diff t\right) {-}{\cal P}_{\kqu}\left(\left(\rule{0mm}{4mm}0,\eta N \widetilde{U}_N(t{-})\right),\diff t\right),
\end{cases}
\]
with $\widetilde{X}_N(0){=}(r,\ell,q_N,u)$, where $\underline{c}_u=C_U/2$, $\underline{c}_m{=}C_M/2$ and $\overline{c}_u{=}2C_U$

Using these SDEs and Relations~\eqref{SDEwo} and with a simple, but careful, induction  on the sequence of instants of jumps of the process $(X_N(t),\widetilde{X}_N(t))$, one can show that, if $N$ is sufficiently large, the relations
\begin{equation}\label{eqcoup}
R_N(t)\le \widetilde{R}_N(t),\quad L_N(t)\le \widetilde{L}_N(t),\quad  U_N(t) \le \widetilde{U}_N(t), \quad Q_N(t) \ge \widetilde{Q}_N(t),
\end{equation}
hold for all $t{\le}\tau_N$. 

\bigskip

\noindent
{\sc Step 2: Analysis of $(\widetilde{X}_N(t))$}\\
With the same method used in the proof of Proposition~\ref{Prop0}, we obtain that the sequence of occupation measures $\widetilde{\mu}_N$ of $(\widetilde{R}_N(t),\widetilde{L}_N(t),\widetilde{U}_N(t))$ is converging in distribution to $\widetilde{\mu}_\infty{=}\diff t{\otimes}\widetilde{\pi}$, where $\widetilde{\pi}$ is the invariant distribution of the Markov process on $\N^3$ whose $Q$-matrix is given by, for $x{=}(r,\ell,u){\in}\N^3$,
  \[
x{\longrightarrow} x{+}
\quad
\begin{cases}
\e{R}{+}\e{U}{-}\e{L}, &\klr(\overline{c}_u{-}\eta),\\
\e{R}{+}\e{U}, &\klr \eta,
\end{cases}
\quad
\begin{cases}
{-}\e{R},&\kri (\underline{c}_m{-}1)r,\\
{-}\e{U},&\kqu u,\\
\e{L},&\kil.
\end{cases}
\]
Furthermore the first moments of $\widetilde{\pi}$ are finite and the convergence of $\croc{\widetilde{\mu}_N,f{\otimes}h}$ to $\croc{\widetilde{\mu}_\infty,f{\otimes}h}$ holds for all $f{\in}{\cal C}_c(\R_+)$ and $h{\in}L_1(\pi)$.

The SDE for $(\widetilde{Q}_N(t))$ gives the relation
\begin{equation}\label{eqs1}
\widetilde{Q}_N(t)=q_N+\widetilde{M}_N(t)+\kil Nt-\kqu\int_0^t \widetilde{Q}_N(s)\widetilde{U}_N(s)\diff s -\kq0\int_0^t  \widetilde{Q}_N(s)\diff s,
\end{equation}
where $(\widetilde{M}_N(t))$ is a martingale whose previsible increasing process is given by
\[
\left(\croc{\widetilde{M}_N}(t)\right)=\left(\kil Nt+\kqu\int_0^t \widetilde{Q}_N(s)\widetilde{U}_N(s)\diff s+\kq0\int_0^t  \widetilde{Q}_N(s)\diff s\right).
\]
From the SDE for the process $(\widetilde{Q}_N(t))$, one has the relation, for $t{\ge}0$, 
\[
\widetilde{Q}_N(t)\leq q_N{+}{\cal P}_{\k0q}\left(\left(\rule{0mm}{4mm}0,N\right),[0,t]\right).
\]
We get that, for any $\eps{>}0$, there exists $K_0$ such that, for $N$ sufficiently large, the relation $\P({\cal F}_N^1){\ge}1{-}\eps$ holds where ${\cal F}_N^1$ is the event 
\begin{equation}\label{eqr52}
{\cal F}_N^1\steq{def} \left\{\sup_{s\leq T}\frac{\widetilde{Q}_N(s)}{N}\le K_0\right\}.
\end{equation}
Let
\[
H_N=\inf\left\{t{\ge}0: \widetilde{U}_N(t) \ge \sqrt{N} \text{ or } \widetilde{Q}_N(t)\ge K_0 N\right\},
\]
then, by Doob's Inequality the process $(M_N(t{\wedge}H_N)/N)$ is converging in distribution to $(0)$. Proposition~\ref{Prop0} show that the sequence $\P(H_N{\le}T)$ is converging to $0$, and therefore $(M_N(t))$ is converging in distribution to $(0)$ on the time interval $[0,T]$.

For $s{\le}t\leq T$ and $\delta{>}0$, Relation~\eqref{eqs1} gives that, on the event ${\cal F}_N^1$, the relation
\begin{multline}\label{eqs2}
\sup_{\substack{s{\leq}t\leq T\\|s-t|\leq \delta}} \left|\frac{\widetilde{Q}_N(t)}{N}{-}\frac{\widetilde{Q}_N(s)}{N}\right|\\\leq
2\sup_{s\leq T} |\widetilde{M}_N(s)| +\kil Nt+\kqu K_0 \sup_{\substack{s{\leq}t\leq T\\|s-t|\leq \delta}}\frac{1}{N}\int_s^t \widetilde{U}_N(u)\diff u
+\kq0 K_0\delta
\end{multline}
holds.  Relation~\eqref{ErT} of Proposition~\ref{Prop0} gives that the sequence of processes
\[
\left(\frac{1}{N}\int_0^t \widetilde{U}_N(u)\diff u\right)
\]
is converging in distribution and is therefore tight.  Using the criterion of the modulus of continuity, see Theorem~7.3 of~\citet{Billingsley}, Relation~\eqref{eqs2} gives  that the sequence of processes $(\widetilde{Q}_N(t)/N)$ is tight and, due to the convergence of $(\widetilde{\mu}_N)$, if $(\widetilde{q}(t))$ is a limiting point  of  $(\widetilde{Q}_N(t)/N)$ then it satisfies the relation 
\[
\widetilde{q}(t)=q_0{+}\kil t-\left(\kqu\int_{\N^3}u\widetilde{\pi}(\diff \ell, \diff q, \diff u){-}\kq0\right) \int_0^t \widetilde{q}(s) \diff s,
\]
which implies the convergence in distribution of $(\widetilde{Q}_N(t)/N)$. We conclude from the convergence of $(\widetilde{\mu}_N,(\widetilde{Q}_N(t)/N))$ and Proposition~\ref{Prop0}  that there exists some $0{<}t_0{\le}T$ and $\eta_0{>}0$ such that if  $q_0{>}\eta_0$ and 
\begin{equation}\label{eqr1}
{\cal F}_N^2\steq{def}\left\{\sup_{s\le t_0} \frac{\widetilde{Q}_N(s)}{N}\ge \eta_0, \sup_{s\le t_0} \frac{\widetilde{R}_N(s){+}\widetilde{L}_N(s){+}\widetilde{U}_N(s)}{\sqrt{N}}\le 1\right\},
\end{equation}
then the relation
\[
  \lim_{N\to+\infty}\P\left({\cal F}_N^2 \right)=1
  \]
  holds. In particular, we have
  \[
  \lim_{N\to+\infty} \P\left(\tau_N\le t_0\right)=0,
  \]
  the coupling relation~\eqref{eqcoup} holds on the time interval $[0,t_0]$ with high probability. 
  From now on it is assumed that $q_0{>}\eta_0$.

\bigskip
\noindent
{\sc Step 3: Convergence}\\
In the same way as in Step~2, with Relation~\eqref{eqr1}, we get that, for any $\eps{>}0$, there exists $K_0$ such that, for $N$ sufficiently large, the relation $\P({\cal E}_N){\ge}1{-}\eps$ holds with 
\begin{equation}\label{eqr5}
{\cal E}_N\steq{def} \left\{\frac{Q_N(t)}{N}\in [\eta_0,K_0], \forall t{\in}[0,t_0]\right\}.
\end{equation}
As in the proof of Proposition~\ref{Prop0},  on the event ${\cal E}_N$,  there exists a coupling such that $(U_N(t))$ is  upper bounded by a process $(Z_1(Nt))$, where $(Z_1(t))$ is an $M/M/\infty$ process on the time interval $[0,t_0]$.

We has seen that the coupling relation~\eqref{eqcoup} on the time interval $t_0$ holds with high probability.  Since the sequence of occupation measures of $(\widetilde{R}_N(t),\widetilde{L}_N(t),\widetilde{U}_N(t))$ is converging in distribution, the sequence of occupation measures $(\Lambda_N)$  of the processes $(R_N(t),L_N(t),U_N(t))$, restricted to $[0,t_0]{\times}\N^3$ is therefore tight. From Relation~\eqref{eqr1} and Proposition~\ref{Prop1},  any limiting point $\Lambda_\infty$  of this sequence  can be represented as, for $g{\in}{\cal C}_c(\R_+{\times}\N^3)$, 
\begin{equation}\label{eqr4}
\croc{\Lambda_\infty,g}=\int_0^{+\infty} \int_{\N^3}g(s,a)\pi_s(\diff a) \diff s,
\end{equation}
almost surely, where $(\pi_s)$ is an optional process with values in ${\cal P}(\N^3)$. Furthermore the convergence to $\Lambda_\infty$, along some subsequence,  holds for a function $g$ of the type $f{\otimes}h$, for $f{\in}{\cal C}_c(\R_+)$ and $h{:}\N^3{\to}\R_+$, such that $(|h(a)|/(a_1{+}a_2{+}a_3), a{=}(a_i){\in}\N^3{\setminus}\{0\})$ is bounded .

With Relation~\eqref{SDEwo}, we have, for $t{\ge}0$,
\begin{multline}\label{eqr3}
\overline{Q}_N(t)\steq{def}\frac{Q_N(t)}{N}=\frac{q_N}{N}+M^0_N(t)+ \k0q t
\\{-}\kqu\int_0^t\overline{Q}_N(u)U_N(u),\diff u
{-}\kq0\int_0^t\overline{Q}_N(u)\diff u,
\end{multline}
where $(\overline{M}_N(t)$ is a martingale.

By using the fact that the event ${\cal F}_N^2{\cap}{\cal E}_N$ has a probability arbitrarily close to $1$ when $N$ gets large , in the same way as in Step~2, we can show that  the martingale  $(\overline{M}_N(t)$ is vanishing and, with the criterion of the modulus of continuity, that  the sequence of processes $(\overline{Q}_N(t))$ is tight on the time interval $[0,t_0]$, and if $(q(t))$ is a limiting point, it is continuous on $[0,t_0]$ and lower bounded by $\eta_0{>}0$. By taking a convenient subsequence $(N_k)$ it can be assumed that $(\Lambda_{N_k},(\overline{Q}_{N_k}(t)))$ is converging in distribution to $(\Lambda_\infty,(q(t)))$ on the time interval $[0,t_0]$.

As in the proof of Proposition~\ref{Prop0}, by using the SDEs~\eqref{SDEwo} and by writing an equation for $h(R_N(t),L_N(t),U_N(t))$, for some function $h$ on $\N^3$ with finite support, we obtain that the measure $\Lambda_\infty$ can be almost surely be represented by Relation~\eqref{eqr4} where, for $s{\in}[0,t_0]$, $\pi_s{\in}{\cal P}(\N^3)$ is the invariant distribution of the Markov process on $\N^3$ whose $Q$-matrix is given by, for $x{=}(r,\ell,u){\in}\N^3$,
  \[
x{\longrightarrow} x{+}
\begin{cases}
\e{L},& \kil,\\
\e{R}{+}\e{U}{-}\e{L}, &\klr C_U \ell,
\end{cases}
\qquad
x{\longrightarrow} x{+}
\begin{cases}
{-}\e{R},&\kri (C_M{-}1)r,\\
{-}\e{U},&\kqu q(s)u.
\end{cases}
\]
This $3$-dimensional process can be seen as a network of three  $M/M/\infty$ queues, for the nodes $L$, $R$ and $U$,  with the following specific dynamic: When a job leaves  node $L$, it generates two simultaneous arrivals at node $R$ and $U$ respectively. See Figure~\ref{FigF1}. 
\begin{center}
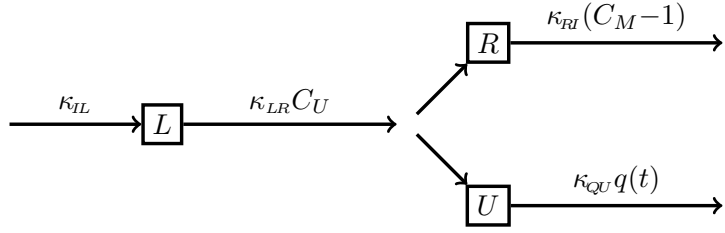
\begin{figure}[H]
\resizebox{10cm}{!}{%
\begin{tikzpicture}[->,node distance=10mm]

\node(M) at (-4,0){};
\node[black, very thick,rectangle,draw](L) at (-2,0){$L$};
\node(L0) at (1,0){};
\node[black, very thick, rectangle,draw](R) at (2,1){$R$};
\node (R0) at (5,1){};
\node[black, very thick, rectangle,draw](U) at (2,-1){$U$};
\node (U0) at (5,-1){};

\path (M) edge [black,very thick,left,midway,above] node {$\kil$} (L);
\path (L) edge [black,very thick,left,midway,above] node {$\klr C_U$} (L0);
\path (L0) edge [black,very thick,left,midway,above] node {} (R);
\path (L0) edge [black,very thick,left,midway,above] node {} (U);
\path (U) edge [black,very thick,left,midway,above] node {$\kqu q(t)$} (U0);
\path (R) edge [black,very thick,left,midway,above] node {$\kri(C_M{-}1)$} (R0);
\end{tikzpicture}}
\caption{Stable Case: $M/M/\infty$ network for fast processes}\label{FigF1}
\end{figure}
\end{center}

By Proposition~\ref{FastInv}, for $f{\in}{\cal C}([0,t_0])$,  almost surely the relation
\[
\int_0^{t_0}\int_{\N^3} f(s) u \pi_s(\diff r, \diff \ell, \diff u)=\int_0^{t_0}\int_{\N^3} f(s) \frac{\kil }{\kqu  q(s)}\diff s
\]
holds. Recall that $q(s){\ge}\eta_0$ for all $s{\leq}t_0$. See Relation~\eqref{eqr5}. 

By taking the limit in Relation~\eqref{eqr3} along the subsequence $(N_k)$, one obtains that the relation
\[
q(t)=q_0+(\k0q -\kil)q t-\kq0\int_0^t q(s)\diff s
\]
holds almost surely for $t{\le}t_0$, and therefore the convergence in distribution of $(\Lambda_N(t),(\overline{Q}_N(t)))$
to $(\Lambda_\infty,(q(t)))$ on the time interval $[0,t_0]$. The fixed point of the above ODE is $q_\infty{=}1{-}\kil/\k0q$. The condition $\kil{<}\k0q$ gives that  if $q_0{\in}(\eta_0,q_\infty)$ then $q(t_0){>}\eta_0$, and if $\eta_0{\ge}q_\infty$ then $q(t_0){\geq}q_\infty$. From this instant, we can obtain the convergence result on the time interval $[t_0,2t_0]$ and, by induction, on $[0,T]$. The proposition is proved. 
\end{proof}

The following corollary is quite intuitive, it just states that in this stable setting the production rate of $P$-particles is maximal.
Consequently, a fraction $\k0q{-}\kil$ of $Q$-particles is degraded by Reaction~\eqref{Ress-eq}. 
\begin{corollary}[Production of $P$-particles]\label{Corol-Stab}
  Under the assumptions of Theorem~\ref{Stab-0}, then for the convergence in distribution
  \[
  \lim_{N\to+\infty}\left(\frac{P_N(t)}{N}\right)=(\kil t)
  \]
 where $(P_N(t))$ is the process defined by Relation~\eqref{Prott}
\end{corollary}
\begin{proof}
  By using Doob's Inequality, it is not difficult to prove that
\[
\left( \frac{P_N(t)}{N}{-}   \klr\int_0^tL_N(s)\left(\frac{U^0_N}{N}{-}\frac{U_N(s)}{N}\right)\diff s\right)
  \]
  is a martingale converging in distribution to $(0)$.

  The convergence of the occupation measure $(\Lambda_N)$ of Theorem~\ref{Stab-0} and the fact that $(U_N(t)/N)$ converges in distribution to $(0)$   give the relation
\begin{multline*}
  \lim_{N\to+\infty}\left( \klr\int_0^tL_N(s)\left(\frac{U^0_N}{N}{-}\frac{U_N(s)}{N}\right)\diff s\right)\\
  =  \left( \klr\int_0^t\int_{\N} b C_U \Pois{\frac{\kil }{\klr C_U}} (\diff b)\diff s\right)=(\kil t)
\end{multline*}

\end{proof}
\begin{theorem}[Under-Loaded System]\label{Sub-0}
Under the scaling assumptions~\eqref{Scal},   if $\k0q{<}\kil$ and  $x_N(0){=}(r,\ell_N,q,U^0_N{-}u)$, with $r$, $q$, $u{\in}\N$ and
  \[
  \lim_{N\to+\infty} \frac{\ell_N}{N}=\ell_0{\in}(0,1),
  \]
  and  $\Lambda_N{\in}{\cal M}(\N^3)$ is the occupation measure of $(R_N(t),Q_N(t),U^0_{N}{-}U_N(t))$, then, for the convergence in distribution,
  \[
  \lim_{N\to+\infty} \left(\left(\frac{L_N(t)}{N}\right),\Lambda_N\right)=((\ell(t)),\Lambda_\infty),
  \]
  where $(\ell(t))$ is the solution of the ODE,  $\dot{\ell}(t){=}\kil (1{-}\ell(t)){-}\k0q $,  starting at $\ell_0$, and $\mu_\infty$ is the measure defined by
\begin{multline*}
 \croc{\Lambda_\infty,F} =\int_0^{+\infty} \int_{\N^3} F(t,r,q,v)\times\\ \Pois{\frac{\k0q }{\kri(C_M{-}1)}} (\diff r)
      \Pois{\frac{\k0q }{\kqu  C_U}} (\diff q)
      \Pois{\frac{\k0q}{\klr  \ell(t)}} (\diff v)\diff t,
\end{multline*}
for any function $F{\in}{\cal C}_c(\R_+{\times}\N^3)$,  .
\end{theorem}
\begin{proof}
The proof follows the same lines as the proof of Theorem~\ref{Stab-0}. We just give a  description of the coupling used, of the stopping time and finally of the associated fast processes. 

  First we define $(V_N(t)){=}(U^0_N{-}V_N(t))$ and the stopping time to consider in this case is 
  \[
\tau_N=\inf\left\{t{\ge}0: V_N(t){\ge}\eta N, \frac{L_N(t)}{N}{\le}\eta\right\}.
\]
An auxiliary process $(\widetilde{X}_N(t)){=}(\widetilde{R}_N(t),\widetilde{L}_N(t),\widetilde{Q}_N(t),\widetilde{V}_N(t))$ is introduced as the solution of the SDE
\[
\begin{cases}
\diff \widetilde{R}_N(t)&={\cal P}_{\klr}\left(\left(\rule{0mm}{4mm}0,N\widetilde{V}_N(t{-})\right),\diff t\right){-}{\cal P}_{\kri}\left(\left(\rule{0mm}{4mm}0,(M^0_N{-}N)\widetilde{R}_N(t{-})\right),\diff t\right),\\
  \diff \widetilde{L}_N(t)&={\cal P}_{\kil}\left(\left(\rule{0mm}{4mm}0,N{-}\widetilde{R}_N(t{-}){-}\widetilde{L}_N(t{-})\right),\diff t\right) \\
&\hspace{2cm}{-}{\cal P}_{\klr}\left(\left(\rule{0mm}{4mm}0,\widetilde{L}_N(t{-})\widetilde{V}_N(t{-})\right),\diff t\right),\\
  \diff \widetilde{Q}_N(t)&={\cal P}_{\k0q}\left(\left(\rule{0mm}{4mm}0,N\right),\diff t\right){-}{\cal P}_{\kqu}\left(\left(\rule{0mm}{4mm}0,\widetilde{Q}_N(t{-})(U^0_N{-}\eta N)\right),\diff t\right),\\
\diff \widetilde{V}_N(t)&={\cal P}_{\kqu}\left(\left(\rule{0mm}{4mm}0,\widetilde{Q}_N(t{-})U^0_N\right),\diff t\right){-}{\cal P}_{\klr}\left(\left(\rule{0mm}{4mm}0,\eta N\widetilde{V}_N(t{-})\right),\diff t\right),
\end{cases}
\]
with the initial condition $(R_N(0),L_N(0),Q_N(0),V_N(0))$.

Using these SDEs and Relations~\eqref{SDEwo} and with an induction on the sequence of instants of jumps of the process $(X_N(t),\widetilde{X}_N(t))$, one can show that the relations
\begin{equation}\label{eqcoup2}
R_N(t)\le \widetilde{R}_N(t),\quad L_N(t)\ge \widetilde{L}_N(t), \quad Q_N(t) \le \widetilde{Q}_N(t), \quad  V_N(t) \le \widetilde{V}_N(t),
\end{equation}
hold for all $t{\le}\tau_N$. 

The tightness of the process  $(\widetilde{R}_N(t),\widetilde{Q}_N(t),\widetilde{V}_N(t))$ is a straightforward consequence of Proposition~\ref{Prop0}. The tightness of $(\widetilde{L}_N(t)/N)$ follows from standard stochastic calculus along the same lines of the proof of Proposition~\ref{Prop1}.

Coming back to our process, the corresponding SDEs are
\[
\begin{cases}
  \diff R_N(t)&={\cal P}_{\klr}\left(\left(\rule{0mm}{4mm}0,L_N(t{-})V_N(t{-})\right),\diff t\right)\\
 &\hspace{2cm}{-}{\cal P}_{\kri}\left(\left(\rule{0mm}{4mm}0,(M^0_N{-}(N{-}R_N(t{-})))R_N(t{-})\right),\diff t\right),\\
  \diff Q_N(t)&={\cal P}_{\k0q}\left(\left(\rule{0mm}{4mm}0,N\right),\diff t\right) \\
 &\hspace{5mm} {-}{\cal P}_{\kqu}\left(\left(\rule{0mm}{4mm}0,Q_N(t{-})(U^0_N{-}V_N(t{-}))\right),\diff t\right)
{-}{\cal P}_{\kq0}\left(\left(\rule{0mm}{4mm}0,Q_N(t{-})\right),\diff t\right),\\
\diff V_N(t)&={\cal P}_{\kqu}\left(\left(\rule{0mm}{4mm}0,Q_N(t{-})(U^0_N{-}V_N(t{-}))\right),\diff t\right)\\
 &\hspace{2cm}{-}{\cal P}_{\klr}\left(\left(\rule{0mm}{4mm}0,L_N(t{-})V_N(t{-})\right),\diff t\right),
\end{cases}
\]
and the SDE for $(L_N(t))$ is 
\[
  \diff L_N(t)={\cal P}_{\kil}\left(\left(\rule{0mm}{4mm}0,N{-}R_N(t{-}){-}L_N(t{-})\right),\diff t\right) 
{-}{\cal P}_{\klr}\left(\left(\rule{0mm}{4mm}0,L_N(t{-})V_N(t{-})\right),\diff t\right).
\]
For $x{=}(r,q,v){\in}\N^2$, the $Q$-matrix associated to the fast process at time $t$,
  \[
x{\longrightarrow} x{+}
\begin{cases}
\e{R}-\e{V}, &\klr r\ell(t),\\
{-}\e{R},&\kri (C_M{-}1)r,\\
\end{cases}
\qquad
x{\longrightarrow} x{+}
\begin{cases}
\e{Q},& \k0q,\\
\e{V}{-}\e{Q},& \kqu C_U q.
\end{cases}
\]
The associated $M/M/\infty$ network is simple in this case. It consists of three queues in series. The invariant distribution is the product of the three Poisson distributions of the statement of the theorem.  See Figure~\ref{FigF2} below. 
\begin{center}
\begin{figure}[H]
\resizebox{10cm}{!}{%
\begin{tikzpicture}[->,node distance=10mm]

\node(M) at (-4,0){};
\node[black, very thick,rectangle,draw](Q) at (-2,0){$Q$};
\node[black, very thick, rectangle,draw](V) at (2,0){$V$};
\node[black, very thick, rectangle,draw](R) at (5,0){$R$};
\node(R0) at (8,0){};

\path (M) edge [black,very thick,left,midway,above] node {$\k0q$} (Q);
\path (Q) edge [black,very thick,left,midway,above] node {$\kqu C_U$} (V);
\path (V) edge [black,very thick,left,midway,above] node {$\klr \ell(t)$} (R);
\path (R) edge [black,very thick,left,midway,above] node {$\kri(C_M{-}1)$} (R0);
\end{tikzpicture}}
\caption{Under-loaded Case:$M/M/\infty$ network for fast processes}\label{FigF2}
\end{figure}
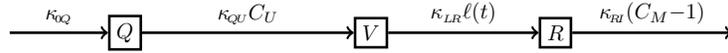
\end{center}
Again, the convergence result is obtained on a time interval $[0,t_0]$, as before the condition  $\k0q{<}\kil$ gives that the solution of the ODE $(\ell(t))$ is converging to its positive fixed point, the same procedure can be used starting at time $t_0$ on the time interval $[t_0,2t_0]$, and by induction on the half line. The theorem is proved. 
\end{proof}
The result for the production of $P$-particles is given by the following corollary.
\begin{corollary}\label{Corol-Sub}
  Under the assumptions of Theorem~\ref{Sub-0}, then for the convergence in distribution
  \[
  \lim_{N\to+\infty}\left(\frac{P_N(t)}{N}\right)=(\k0q t)
  \]
 where $(P_N(t))$ is the process defined by Relation~\eqref{Prott}
\end{corollary}
\begin{proof}
We proceed in the same way as in the proof of Corollary~\ref{Corol-Stab}, by noting that the relations
  \begin{multline*}
  \lim_{N\to+\infty}\left(\frac{P_N(t)}{N}\right)=  \lim_{N\to+\infty}\left( \klr\int_0^t\overline{L}_N(s)\left(U^0_N{-}U_N(s)\right)\diff s\right)\\
  =  \left( \klr\int_0^t\int_{\N} \ell(s) v \Pois{\frac{\k0q }{\klr \ell(s)}} (\diff v)\diff s\right)=(\k0q t)
  \end{multline*}
  hold. 
\end{proof}
\section{A Model with Regulation}\label{Reg-1-Sec}
We return to the model of Section~\ref{Model-Sec}, the process
\[
(X_N(t)){=}(S_N(t),R_N(t),L_N(t),Q_N(t),U_N(t))
\]
is defined by the SDEs of Section~\ref{SDE-Sec}. 

\subsection{Stable Network}\label{R1sec}\
\addcontentsline{toc}{section}{\ref{R1sec}. Stability Regime}

In this case the flow of resources, $Q$-particles,  is sufficiently important. As it can be expected, the regulation  mechanism does not have a significant impact. The following result is essentially identical to Theorem~\ref{Stab-0}. The proof is skipped for this reason.

Intuitively, starting from a convenient subset of initial states, a free $R$-particle is paired to an $M$-particle at a rate at least $M_M{-}N{\sim}(C_M{-}1)N$,  the sequestration rate  being $\krs$, sequestration is unlikely, given that there is a finite number $R$-particles.  The process $(S_N(t))$ is not in fact a fast process,   in the limit for the occupation measure, it is a Dirac measure at $0$. 
\begin{theorem}[Stability]\label{Stab-1}
Under the scaling assumptions~\eqref{Scal},  if $\k0q{>}\kil $, and the initial state is $x_N(0)=(s,r,\ell,q_N,u){\in}{\cal S}_N$, such that
  \[
  \lim_{N\to+\infty} \frac{q_N}{N}=q_0{>}0,
  \]
then, for the convergence in distribution,
  \[
  \lim_{N\to+\infty} \left(\frac{Q_N(t)}{N}\right)=(q(t)),
  \]
  where $(q(t))$ is the solution of the ODE,
  \[
  \dot{q}(t){=}\k0q {-}\kil {-}\kq0 q(t),
  \]
starting at $q_0$ and the occupation measure of $(S_N(t),R_N(t),L_N(t),U_N(t))$ is converging in distribution to $\mu_\infty$ defined by, for  $F{\in}{\cal C}_c(\R_+{\times}\N^4)$,
\[
\croc{\mu_\infty,F}=\croc{\Lambda_\infty,\widetilde{F}},
\]
where $\widetilde{F}(t,(x)){=}F(t,(0,x))$, for $t{\ge}0$ and $x{\in}\N^3$, and $\Lambda_\infty$ is the measure on $\R_+{\times}\N^3$ defined by Relation~\eqref{eqr0}.
\end{theorem}

For the production of $P$-particles, it is easily seen that Corollary~\ref{Corol-Stab} also holds. The sequestered $R$-particles do not play any role
  in this case. 
\subsection{Under-Loaded Network: Optimal Sequestration}\label{R2sec}
\addcontentsline{toc}{section}{\ref{R2sec}. Under-Loaded Regime: Optimal Sequestration}

In this section the efficiency of the regulation mechanism is shown in the following theorem. When the input rate of $Q$-particles is too low, $\k0q{<}\kil $, if an additional condition is satisfied, Relation~\eqref{Cond3}, the  process $(L_N(t))$ is still $O(1)$. Only the processes $(S_N(t))$ and $(U_N(t))$ are of the order of $N$. This is in contrast with Theorem~\ref{Sub-0} under the same condition when there is no regulation.

This case stands out also for another reason. For this regime even if the coordinates of the initial state are all of the correct asymptotic orders of magnitude, it may happen that the Markov process may leave this part of the state space after some finite time, one of the $O(1)$ coordinates become of the order of $N$. Heuristically, this happens for example when there are too many sequestered $R$-particles initially. The remaining $R$-particles cannot cope with the (small) incoming flow of $Q$-particles, so the process $(Q_N(t))$ will be of the order of $N$ on a time interval before returning to the correct order of magnitude $O(1)$.

In the five asymptotic results of this paper, this is the only case when this phenomenon happens. The introduction of the neighborhood $V_0$ for the initial state is a way of preserving the orders of magnitudes on the whole time axis. This is related to the stability properties of a dynamical system in $\R^2$. 

\begin{theorem}\label{Stab-1-Max}
Let 
\begin{equation}\label{Fix3}
  \begin{cases}
s_\infty&=\displaystyle 1{-}\frac{\k0q}{\kil}, \\
u_\infty &=\displaystyle \frac{\ksr}{\krs}\frac{\kri}{\k0q}\left(C_M{-}\frac{\k0q}{\kil}\right)\left(1{-}\frac{\k0q}{\kil}\right),
  \end{cases}
\end{equation}
and assume that the initial state of $(X_N(t))$  is  $x_N(0){=}(s_N,r,\ell,q,u_N){\in}\N^5$, with 
  \[
  \lim_{N\to+\infty} \left(\frac{s_N}{N},\frac{u_N}{N}\right)=(s_0,u_0),
  \]
  and   $\Lambda_N$ is the  occupation measures of $(R_N(t),L_N(t),Q_N(t))$.

Under  the scaling assumptions~\eqref{Scal} and  the conditions
  \begin{equation}\label{Cond3}
\k0q{<}\kil \text{ and }\frac{\ksr}{\krs }\frac{\kri }{\k0q}\left(C_M{-}\frac{\k0q}{\kil}\right)\left(1{-}\frac{\k0q}{\kil}\right) < C_U,
  \end{equation}
 there exists a neighborhood of $V_0$ of $(s_\infty,u_\infty)$ in $(0,1){\times}(0,C_U)$ such that if  $(s_0,u_0){\in}V_0$, then the relation
\[
  \lim_{N\to+\infty} \left(\left(\frac{S_N(t)}{N},\frac{U_N(t)}{N}\right),\Lambda_N\right)=((s(t),u(t)),\Lambda_\infty),
\]
holds  for the  convergence in distribution, where $(s(t), u(t))$ is the solution of the ODE,
 \begin{equation}\label{ODE3}
   \begin{cases}  
\dot{s}(t)&{=}\displaystyle \krs u(t)\frac{\kil(1{-}s(t)){+}\ksr s(t)}{\kri (C_M{-}(1{-}s(t))){+}\krs u(t)}-\ksr s(t),\\
\dot{u}(t)&{=}\displaystyle\kil(1{-}s(t)){-}\k0q,
\end{cases}
 \end{equation}
starting at $(s_0,u_0)$, and $\Lambda_\infty{\in}{\cal M}(\R_+{\times}\N^3)$ is the measure defined by
\begin{multline}\label{FPeq3}
\croc{\Lambda_\infty,F}=\int_0^{+\infty} \int_{\N^3} F(t,r,l,q)
\Pois{\frac{\kil(1{-}s(t)){+}\ksr s(t)}{\kri (C_M{-}(1{-}s(t))){+}\krs u(t)}} (\diff r)\\
\Pois{ \frac{\kil(1{-}s(t))}{\klr (C_U{-}u(t))}} (\diff \ell)
\Pois{\frac{\k0q}{\kqu u(t)}} (\diff q)\,\diff t,
\end{multline}
for any function $F{\in}{\cal C}_c(\R_+{\times}\N^3)$.

The point $(s_\infty,u_\infty)$ is a locally stable point of the dynamical system $(s(t),u(t))$. 
\end{theorem}
\begin{proof}
For the beginning of the proof,  we proceed as in the proof of Theorem~\ref{Sub-0} and by following the same method used to prove Theorem~\ref{Stab-0}: First define a convenient stopping time, then introduce an auxiliary process $(\widetilde{X}_N(t))$ and a coupling with the process $(X_N(t))$. Tightness/convergence properties for $(\widetilde{X}_N(t))$ are then used to prove tightness and convergence results for $(X_N(t))$. 
  
We define
  \[
\tau_N=\inf \left\{t{>}0: \min\left(\frac{S_N(t)}{N},\frac{U_N(t)}{N},\frac{U^0_N{-}U_N(t)}{N}\right)\le \eta \right\}
\]
and $(\widetilde{S}_N(t),\widetilde{U}_N(t),\widetilde{R}_N(t),\widetilde{L}_N(t),\widetilde{Q}_N(t))$ the solution of the SDE
\begin{equation}
\begin{cases}
\diff \widetilde{S}_N(t)&= {-}{\cal P}_{\ksr}\left(\left(\rule{0mm}{4mm}0,\widetilde{S}_N(t{-})\right),\diff t\right),\\
\diff \widetilde{U}_N(t)&={-}{\cal P}_{\kqu}\left(\left(\rule{0mm}{4mm}0,\widetilde{Q}_N(t{-})\widetilde{U}_N(t{-})\right),\diff t\right),\\
\diff \widetilde{R}_N(t)&={\cal P}_{\klr}\left(\left(\rule{0mm}{4mm}0,U^0_N \widetilde{L}_N(t{-})\right),\diff t\right) 
{+}{\cal P}_{\ksr}\left(\left(\rule{0mm}{4mm}0,N\right),\diff t\right)\\
&\hspace{2cm}{-}{\cal P}_{\kri}\left(\left(\rule{0mm}{4mm}0,\left(M^0_N{-}N\right)\widetilde{R}_N(t{-})\right),\diff t\right),\\
\diff \widetilde{L}_N(t)&={\cal P}_{\kil}\left(\left(\rule{0mm}{4mm}0,N\right),\diff t\right) 
{-}{\cal P}_{\klr}\left(\left(\rule{0mm}{4mm}0,\eta N\widetilde{L}_N(t{-})\right),\diff t\right),\\
  \diff \widetilde{Q}_N(t)&={\cal P}_{\k0q}\left(\left(\rule{0mm}{4mm}0,N\right),\diff t\right) 
{-}{\cal P}_{\kqu}\left(\left(\rule{0mm}{4mm}0,\eta N \widetilde{Q}_N(t{-})\right),\diff t\right),
\end{cases}
\end{equation}
with initial point $(s_N,u_N,r,l,q){\in}{\cal S}_N$. 
\begin{multline*}\label{eqcoup3}
 S_N(t)\ge \widetilde{S}_N(t), \quad U_N(t) \ge \widetilde{U}_N(t), \\
  R_N(t)\le \widetilde{R}_N(t),\quad L_N(t)\le \widetilde{L}_N(t),   \quad Q_N(t) \le \widetilde{Q}_N(t),
\end{multline*}
hold for all $t{\le}\tau_N$.

From there, by proceeding as before,  it is not difficult to prove the convergence in distribution of the occupation measure of the sequence of occupation measures of $(\widetilde{R}_N(t),\widetilde{L}_N(t),\widetilde{Q}_N(t))$ and of the scaled processes $(\widetilde{S}_N(t)/N,\widetilde{L}_N(t)/N)$. This convergence gives the existence of $t_0$ such that
\[
\lim_{N\to+\infty}\P(\tau_N\ge t_0)=1,
\]
and, consequently, the tightness property of $(\Lambda_N,(\overline{X}_N(t)))$ on the time interval $[0,t_0]$.

For $x{=}(r,\ell,q){\in}\N^2$, the $Q$-matrix associated to the fast process $(R_N(t),L_N(t),Q_N(t))$  at time $t$ is
\[
x{\longrightarrow} x{+}
\begin{cases}
\e{R}-\e{L}, &\klr (C_U{-}u(t))\ell,\\
\e{R}, & \ksr s(t),\\
{-}\e{R}, & (\krs u(t) {+} \kri(C_M{-}1{+}s(t)))r,
\end{cases}
\quad
x{\longrightarrow} x{+}
\begin{cases}
\e{L},& \kil (1{-}s(t)),\\
\e{Q},& \k0q,\\
{-}\e{Q},& \kqu u(t)q.
\end{cases}
\]
The associated $M/M/\infty$ network is depicted in Figure~\ref{FigF3} below. It is easily checked that the product of the three Poisson distribution in the right-hand side of Relation~\eqref{FPeq3} is indeed the invariant distribution of this Markov process.  In the same way as  in the proof Theorem~\ref{Sub-0} by using tightness and continuity arguments, {\em on the time interval} $[0,t_0]$ we obtain the desired convergence of $(\Lambda_N,(S_N(t)/N,U_N(t)/N))$. 

We now investigate the dynamical system defined by the ODE~\eqref{ODE3}. It is clear that $w_{\infty}{=}(s_\infty,u_\infty)$ is its unique fixed point, and it is in the domain $(0,1){\times}(0,C_U)$ by Condition~\eqref{Cond3}. The Jacobian matrix of $(s(t),u(t))$ at $w_\infty$ is 
\[
\begin{bmatrix}
  -\frac{1}{D_0}\ksr\kil((C_M{-}1)\kil{}^2 {-} \k0q{}^2 {+} \kil{}^2) &\frac{1}{D_0}\krs\k0q{}^2\kil{}^2\\ -\kil & 0,
\end{bmatrix}
\]
with $D_0{=}\k0q\kil {-} \k0q\ksr + \kil\ksr)((C_M{-}1)\kil {-} \k0q {+} \kil$, and its characteristic polynomial is $P_3$ with 
\begin{multline*}
P_3(x)\steq{def}\kri\left(\k0q\kil {+}\ksr(\kil{-}\k0q)\right)((C_M{-}1)\kil{+}\kil{-}\k0q)x^2 \\+ \left((C_M{-}1)\kil^3\kri\ksr{+}\kil\kri\ksr(\kil{}^2{-}\k0q{}^2)\right)x{+}\kil{}^3\krs\k0q{}^2.
\end{multline*}
The condition $\kil{>}\k0q$ implies that all coefficients of $P_3$ are positive. If the zeroes of $P_3$ are real, then they are negative. Otherwise, their real part is negative. This implies that  $w_\infty$ is an asymptotically stable point of $(s(t),u(t))$ and, therefore, that there exists a neighborhood $V_0$ of $w_\infty$ such that if $(u(0),s(0)){\in}V_0$ then $(u(t),s(t)){\in}V_0$  for all $t{\ge}0$. See~\citet{Hirsch}. From there, we can start from $(u(t_0),s(t_0)){\in}V_0$ and by checking carefully the dependence of $t_0$ with respect to $(u(0),s(0))$,  repeat the same operation on the time interval $[t_0,2t_0]$ and concludes the proof of our theorem by induction. 
\end{proof}
\begin{center}
\begin{figure}[H]
\resizebox{10cm}{!}{%
\begin{tikzpicture}[->,node distance=10mm]

\node(M) at (-3,2){};
\node[black, very thick,rectangle,draw](Q) at (0,2){$Q$};
\node(N) at (4,2){};
\node(L0) at (-3,0){};
\node[black, very thick, rectangle,draw](L) at (0,0){$L$};
\node[black, very thick, rectangle,draw](R) at (4,0){$R$};
\node(R0) at (9,0){};

\path (M) edge [black,very thick,left,midway,above] node {$\k0q$} (Q);
\path (Q) edge [black,very thick,left,midway,above] node {$\kqu u(t)$} (N);
\path (L) edge [black,very thick,left,midway,above] node {$\klr (C_U{-}u(t))$} (R);
\path (L0) edge [black,very thick,left,midway,above] node {$\kil (1{-}s(t))$} (L);
\path (R) edge [black,very thick,left,midway,above] node {$\krs u(t){+} \kri(C_M{-}1{+}s(t))$} (R0);
\end{tikzpicture}}
\caption{Under-loaded Regime --- Optimal Sequestration.\\The $M/M/\infty$ network for fast processes}\label{FigF3}
\end{figure}
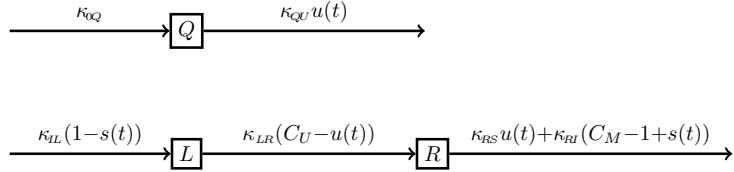
\end{center}

\begin{corollary}[Production of $P$-particles]
Under the assumptions of Theorem~\ref{Stab-1-Max}, then for the convergence in distribution
  \[
  \lim_{N\to+\infty}\left(\frac{P_N(t)}{N}\right)= \left( \kil\int_0^t(1{-}s(w))\diff w\right)
  \]
 where $(P_N(t))$, resp $(s(w))$,  is the process defined by Relation~\eqref{Prott}, resp. Relation~\eqref{ODE3}. 
\end{corollary}
Note that if the initial state $(s_0,u_0)$ is the equilibrium point of the ODE~\eqref{ODE3}, then the limit of $(P_N(t)/N)$ is $(\k0q t)$ as it can be expected. 
\begin{proof}
We proceed in the same way as in the proof of Corollary~\ref{Corol-Stab}, by noting that as a consequence of Theorem~\ref{Stab-1-Max}, the relations 
\begin{multline*}
  \lim_{N\to+\infty}\left( \klr\int_0^tL_N(w)\left(\frac{U^0_N}{N}{-}\frac{U_N(w)}{N}\right)\diff w\right)\\
  =  \left( \klr\int_0^t\int_{\N} \ell (C_U{-}u(w)) \Pois{ \frac{\kil(1{-}s(w))}{\klr (C_U{-}u(w))}} (\diff \ell)\diff w\right)
\\= \left( \kil\int_0^t(1{-}s(w))\diff w\right)
\end{multline*}
  hold. 
\end{proof}
%%%%%%
\subsection{Under-Loaded Network: Saturation}\label{R3sec}
\addcontentsline{toc}{section}{\ref{R3sec}. Under-Loaded Regime: Saturation}

When $\k0q{<}\kil$ but the second condition of Relation~\eqref{Cond3} is not true, then the regulation mechanism cannot avoid the undesirable property that the process $(L_N(t))$ is of the order of $N$. We are in the qualitative situation of Theorem~\ref{Sub-0} without regulation. 

\begin{theorem}\label{Stab-1-Sat}
Under  the scaling assumptions~\eqref{Scal} and  the conditions
\begin{equation}\label{Cond4}
\k0q{<}\kil \text{ and }\frac{\ksr}{\krs }\frac{\kri }{\k0q}\left(C_M{-}\frac{\k0q}{\kil}\right)\left(1{-}\frac{\k0q}{\kil}\right) > C_U,
\end{equation}
assume that  the initial state is  $x_N(0){=}(s_N,r,\ell_N,q,U^0_N{-}u)$, with $r$, $q$, $u{\in}\N$ and
\[
\lim_{N\to+\infty} \left(\frac{s_N}{N},\frac{l_N}{N}\right)=(s_0,\ell_0) \text{ with } l_0, s_0{>}0 \text{ and }s_0{+}\ell_0{<}1,
\]
and   $\Lambda_N{\in}{\cal M}(\N^3)$ denotes the occupation measure of $(R_N(t),Q_N(t),U^0_N{-}U_N(t))$.

For the convergence in distribution, the relation
  \[
  \lim_{N\to+\infty} \left(\left(\frac{S_N(t)}{N},\frac{L_N(t)}{N}\right),\Lambda_N\right)=((s(t),\ell(t)),\Lambda_\infty),
  \]
holds, where $(s(t), \ell(t))$ is the solution of the ODE,
\begin{equation}\label{ODE4}
\begin{cases}  
\dot{\ell}(t)&\displaystyle{=}\kil(1{-}\ell(t){-}s(t))-\k0q,\\
\dot{s}(t)&\displaystyle{=}\krs C_U\frac{\k0q{+}\ksr s(t)}{\kri (C_M{-}1{+}s(t)){+}\krs C_U}-\ksr s(t)
\end{cases}
\end{equation}
starting at $(s_0,\ell_0)$, and $\Lambda_\infty{\in}{\cal M}(\R_+{\times}\N^3)$ is the measure defined by
\begin{multline*}
\croc{\Lambda_\infty,F}=\int_0^{+\infty} \int_{\N^3} F(t,r,q,v)
\Pois{\frac{\k0q{+}\ksr s(t)}{\kri \left(C_M{-}1{+}s(t))\right){+}\krs C_U}} (\diff r)\\
\Pois{\frac{\k0q}{\kqu C_U}} (\diff q)
\Pois{\frac{\k0q}{\klr \ell(t)}} (\diff v)\,\diff t,
\end{multline*}
for any function $F{\in}{\cal C}_c(\R_+{\times}\N^3)$.

The equilibrium point  $(s_\infty,\ell_\infty)$ of $(s(t),\ell(t))$,
  \begin{equation}\label{Fix4}
  \begin{cases}
s_\infty&=\displaystyle \frac{1}{2}\left(-(C_M{-}1)+\sqrt{(C_M{-}1)^2+4C_U\frac{\k0q}{\kri}\frac{\krs}{\ksr}}\right)\\ \ &\\
\ell_\infty &=\displaystyle 1{-}\frac{\k0q}{\kil}-s_\infty,
  \end{cases}
\end{equation}
is asymptotically stable.
\end{theorem}
\begin{proof}
We denote $(V_N(t)){=}(U^0_N{-}U_N(t))$.   The corresponding SDEs are
\[
\begin{cases}
  \diff R_N(t)&={\cal P}_{\klr}\left(\left(\rule{0mm}{4mm}0,L_N(t{-})V_N(t{-})\right),\diff t\right)
  {+}{\cal P}_{\ksr}\left(\left(\rule{0mm}{4mm}0,S_N(t{-})\right),\diff t\right)\\
&\hspace{1cm}{-}{\cal P}_{\krs}\left(\left(\rule{0mm}{4mm}0,R_N(t{-})U_N(t{-})\right)\diff t\right)\\
 &\hspace{1cm}{-}{\cal P}_{\kri}\left(\left(\rule{0mm}{4mm}0,(M^0_N{-}(N{-}R_N(t{-}){-}S_N(t{-})))R_N(t{-})\right),\diff t\right),\\
  \diff Q_N(t)&={\cal P}_{\k0q}\left(\left(\rule{0mm}{4mm}0,N\right),\diff t\right) \\
 &\hspace{5mm} {-}{\cal P}_{\kqu}\left(\left(\rule{0mm}{4mm}0,Q_N(t{-})(U^0_N{-}V_N(t{-}))\right),\diff t\right)
{-}{\cal P}_{\kq0}\left(\left(\rule{0mm}{4mm}0,Q_N(t{-})\right),\diff t\right),\\
\diff V_N(t)&={\cal P}_{\kqu}\left(\left(\rule{0mm}{4mm}0,Q_N(t{-})(U^0_N{-}V_N(t{-}))\right),\diff t\right)\\
 &\hspace{2cm}{-}{\cal P}_{\klr}\left(\left(\rule{0mm}{4mm}0,L_N(t{-})V_N(t{-})\right),\diff t\right),
\end{cases}
\]
and the SDEs for $(S_N(t),L_N(t))$ are
\[
\begin{cases}
  \diff S_N(t)&={\cal P}_{\krs}\left(\left(\rule{0mm}{4mm}0,R_N(t{-})U_N(t{-})\right),\diff t\right) 
        {-}{\cal P}_{\ksr}\left(\left(\rule{0mm}{4mm}0,S_N(t{-})\right),\diff t\right),\\
  \diff L_N(t)&={\cal P}_{\kil}\left(\left(\rule{0mm}{4mm}0,N{-}R_N(t{-}){-}S_N(t{-}){-}L_N(t{-})\right),\diff t\right) \\
        &\hspace{3cm}{\cal P}_{\klr}\left(\left(\rule{0mm}{4mm}0,L_N(t{-})V_N(t{-})\right),\diff t\right).
\end{cases}
\]
We define
  \[
\tau_N=\inf \left\{t{>}0: \min\left(\frac{S_N(t)}{N},\frac{L_N(t)}{N}\right)\le \eta\text{ or }\max\left(\frac{R_N(t)}{N},\frac{U_N(t)}{N},\frac{V_N(t)}{N}\right)\ge \eta \right\}
\]
The auxiliary process $(\widetilde{X}_N(t)){=}(\widetilde{R}_N(t),\widetilde{L}_N(t),\widetilde{Q}_N(t),\widetilde{V}_N(t))$ is the solution of the SDE
\[
\begin{cases}
    \diff \widetilde{S}_N(t)&={\cal P}_{\krs}\left(\left(\rule{0mm}{4mm}0,\widetilde{R}_N(t{-})\widetilde{U}_N(t{-})\right),\diff t\right) 
        {-}{\cal P}_{\ksr}\left(\left(\rule{0mm}{4mm}0,\widetilde{S}_N(t{-})\right),\diff t\right),\\
        \diff \widetilde{R}_N(t)&={\cal P}_{\klr}\left(\left(\rule{0mm}{4mm}0,N\widetilde{V}_N(t{-})\right),\diff t\right)
  {+}{\cal P}_{\ksr}\left(\left(\rule{0mm}{4mm}0,N\right),\diff t\right)\\
 &\hspace{2cm}{-}{\cal P}_{\kri}\left(\left(\rule{0mm}{4mm}0,(M^0_N{-}N)\widetilde{R}_N(t{-})\right),\diff t\right),\\
  \diff \widetilde{L}_N(t)&={\cal P}_{\kil}\left(\left(\rule{0mm}{4mm}0,(1{-}\eta)N{-}\widetilde{L}_N(t{-})\right),\diff t\right) \\
&\hspace{2cm}{-}{\cal P}_{\klr}\left(\left(\rule{0mm}{4mm}0,\widetilde{L}_N(t{-})\widetilde{V}_N(t{-})\right),\diff t\right),\\
  \diff \widetilde{Q}_N(t)&={\cal P}_{\k0q}\left(\left(\rule{0mm}{4mm}0,N\right),\diff t\right){-}{\cal P}_{\kqu}\left(\left(\rule{0mm}{4mm}0,\widetilde{Q}_N(t{-})(U^0_N{-}\eta N)\right),\diff t\right),\\
\diff \widetilde{V}_N(t)&={\cal P}_{\kqu}\left(\left(\rule{0mm}{4mm}0,\widetilde{Q}_N(t{-})U^0_N\right),\diff t\right){-}{\cal P}_{\klr}\left(\left(\rule{0mm}{4mm}0,\eta N\widetilde{V}_N(t{-})\right),\diff t\right),
\end{cases}
\]
with the initial condition  $(S_N(0),R_N(0),L_N(0),Q_N(0),V_N(0))$.

\[
\dot{s}(t){=}-\frac{P_1(s)}
{((C_M{-}1)\kri + C_U\krs + \kri s)},
\]
with
\[
P_1(s)\steq{def}\kri\ksr s^2+(C_M{-}1)\kri\ksr s  - C_U\k0q\krs,
\]
$P_1$ has two roots,
\[
\frac{1}{2}\left(-(C_M{-}1)\pm\sqrt{(C_M{-}1)^2+4C_U\frac{\k0q}{\kri}\frac{\krs}{\ksr}}\right),
\]
only one, $s_\infty$ is positive, the other is negative. It is easily checked that the second condition of Relation~\eqref{Cond4} implies that $s_\infty{\in}(0,1)$. The equilibrium point $\ell_\infty$ of $(\ell(t))$ is positive if and only if $s_\infty{<}1{-}\k0q/\kil$.  This is equivalent to the relation
$P_1(1{-}\k0q/\kil){>}0$ and, with  trite calculations, this is  the condition
\[
\left(C_M{-}\frac{\k0q}{\kil}\right)\left(1{-} \frac{\k0q}{\kil}\right)-C_U\frac{\k0q}{\kri}\frac{\krs}{\ksr} {>}0,
\]
which is exactly the second condition of Relation~\eqref{Cond4}.

If $s_0{\in}(0,s_\infty)$, resp. if $s_0{\in}(s_\infty,1)$, one gets that $(s(t))$ is non-decreasing, resp. $(s(t))$ is non-increasing.

\end{proof}
\begin{center}
\begin{figure}[H]
\resizebox{10cm}{!}{%
\begin{tikzpicture}[->,node distance=10mm]

\node(M) at (-3,0){};
\node[black, very thick,rectangle,draw](Q) at (-2,0){$Q$};
\node(L0) at (2,-1){};
\node[black, very thick, rectangle,draw](V) at (0,0){$V$};
\node[black, very thick, rectangle,draw](R) at (2,0){$R$};
\node(R0) at (8,0){};

\path (M) edge [black,very thick,left,midway,above] node {$\k0q$} (Q);
\path (Q) edge [black,very thick,left,midway,above] node {$\kqu C_U$} (V);
\path (V) edge [black,very thick,left,midway,above] node {$\klr \ell(t)$} (R);
\path (L0) edge [black,very thick,left,midway] node {$\ksr s(t)$} (R);
\path (R) edge [black,very thick,left,midway,above] node {$\krs C_U{+} \kri(C_M{-}1{+}s(t))$} (R0);
\end{tikzpicture}}
\caption{Under-loaded Regime --- Saturation.\\The $M/M/\infty$ network for fast processes}\label{FigF4}
\end{figure}
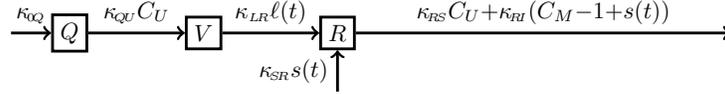
\end{center}

\begin{corollary}
  Under the assumptions of Theorem~\ref{Stab-1-Sat}, then for the convergence in distribution
  \[
  \lim_{N\to+\infty}\left(\frac{P_N(t)}{N}\right)=(\k0q t),
  \]
 where $(P_N(t))$ is the process defined by Relation~\eqref{Prott}
\end{corollary}
\begin{proof}
The proof is similar to the proof of Corollary~\ref{Corol-Sub}.
\end{proof}

\printbibliography

\end{document}